\documentclass[12pt, a4paper]{amsart}

\usepackage{amsmath,amssymb,amsthm,xypic,epsfig }
\usepackage{enumerate}
\usepackage[dvips]{color}
\usepackage{bbm}
\usepackage[colorinlistoftodos]{todonotes}
\usepackage{version}

\DeclareFontFamily{OT1}{rsfs}{}
\DeclareFontShape{OT1}{rsfs}{n}{it}{<-> rsfs10}{}
\DeclareMathAlphabet{\curly}{OT1}{rsfs}{n}{it}

\newtheorem{Thm}{Theorem}[section]
\newtheorem*{MThm}{Main Theorem}
\newtheorem{lem}[Thm]{Lemma}

\newtheorem{cor}[Thm]{Corollary}

\newtheorem{prop}[Thm]{Proposition}
\newtheorem{conj}[Thm]{Conjecture}
\newtheorem{``Conj"}[Thm]{``Conjecture"}

\newtheorem{Claim}[Thm]{Claim}
\newtheorem{Ques}[Thm]{Question}
\newtheorem{Prob}[Thm]{Problem}

\theoremstyle{remark}
\newtheorem{Rem}[Thm]{Remark}
\newtheorem{ex}[Thm]{Example}
\theoremstyle{definition}
\newtheorem{defn}[Thm]{Definition}

\newtheorem{Setup}{Setup}

\newtheorem{ntt}{Notations}

\newcommand{\Spec}{\mathop{\mathrm{Spec}}\nolimits}

\newcommand{\DF}{\mathop{\mathrm{DF}}\nolimits}

\newcommand{\vvol}{\mathop{\widehat{\mathrm{vol}}}\nolimits}

\def\dim{\operatorname{dim}}
\def\diam{\operatorname{diam}}
\def\Ric{\operatorname{Ric}}

\def\det{\operatorname{det}}
\def\vol{\operatorname{vol}}
\def\deg{\operatorname{deg}}

\def\k{\mathbbm{k}}
\def\R{\mathbb{R}}

\def\C{\mathbb{C}}
\def\G{\mathbb{G}}
\def\Z{\mathbb{Z}}
\def\O{\mathcal{O}}
\def\P{\mathbb{P}}
\def\Q{\mathbb{Q}}

\def\A{\mathbb{A}}

\def\X{\mathcal{X}}
\def\Y{\mathcal{Y}}
\def\B{\mathcal{B}}

\def\M{\mathcal{M}}

\def\H{\mathcal{H}}
\def\U{\mathcal{U}}

\excludeversion{memo}

\begin{document}

\title[Moduli of Sasaki-Einstein manifolds]
{Compact 
moduli of Calabi-Yau cones and 
Sasaki-Einstein spaces}
\author{Yuji Odaka}


\maketitle
\thispagestyle{empty}

\begin{abstract}
We construct proper moduli algebraic spaces 
of K-polystable $\Q$-Fano cones 
(a.k.a. Calabi-Yau cones) 
or equivalently their links i.e., 
Sasaki-Einstein manifolds with singularities. 

As a byproduct, it gives alternative algebraic 
construction of proper K-moduli of $\Q$-Fano varieties. 
In contrast to the previous algebraic 
proof of its properness (\cite{BHLLX, LXZ}), we do not use 
the $\delta$-invariants 
(\cite{FO, BJ}) 
nor the $L^2$-normalized 
Donaldson-Futaki invariants. 
We use the local normalized volume of \cite{Li} and 
the higher $\Theta$-stable reduction 
\cite{Od24b} instead. 
\end{abstract}

\tableofcontents


\section{Introduction}\label{sec:intro}

It follows from the recent breakthrough \cite{DS}, 
combined with the Gromov's precompactness theorem and 
the theory of Gromov-Hausdorff distance 
\cite{Gromov}, 
that there should be 
compactifying topological space of moduli 
spaces of K\"ahler-Einstein Fano manifolds, 
letting the boundary parametrizes certain singular 
K\"ahler-Einstein Fano varieties. 
Indeed, \cite{DS} uses Cheeger-Colding theory to 
prove that for 
K\"ahler-Einstein smooth Fano manifolds of fixed dimensions, 
the Gromov-Hausdorff 
limits, which always exists for subsequences, 
admit the structure of log terminal 
K\"ahler-Einstein Fano varieties, 
which are $\Q$-Gorenstein smoothable (cf., \cite[\S 2]{OSS}). 
In particular, they admit the structure of algebraic varieties. 
This also matches the general K-moduli conjecture then (\cite[Conjecture 5.2]{Od10}), a 
purely algebro-geometric conjecture, which was later 
refined in the course of developments. 
Indeed, the obtained compact moduli topological spaces were enhanced to be proper (compact) 
algebraic spaces by using K-stability in 
the series of works 
\cite{OSS, LWXF, SSY, OdF} (cf., also \cite{MM} which predates K-stability). 
By more birational algebro-geometric discussions after 
that, 
recently we witnessed a celebrated algebraic (re-)proof and 
generalization 
of the facts that the moduli (stack) of 
K-polystable Fano varieties satisfies 
the valuative criterion of properness 
by \cite{BHLLX, LXZ}, and the (categorical) good moduli
\footnote{categorical moduli of an algebraic stack means the one in 
e.g., \cite[106.12]{SP} and good moduli 
is in the sense of \cite{Alper}. The author 
confesses that he has often used the term 
``coarse moduli" before, for the exactly same meaning as 
the categorical moduli in older papers, 
i.e., 
not assuming the bijectivity condition for 
geometric points, which is assumed in most of the literature. 
Henceforth, we will avoid terminlogical 
confusion. }
is projective (\cite{CP, XZ}). 
Their algebraic proof depends on the theory of 
$\delta$-invariants \cite{FO, BJ} of Fano varieties. 
This reproves and extends the original construction 
\cite{LWXF, OdF} (cf., also closely related results in \cite{SSY}) 
to generally singular $\Q$-Fano varieties. 

In this paper, we extend this K-moduli construction to that of Fano {\it cones}, e.g., 
contracted (pluri-)anticanonical line bundles over K\"ahler-Einstein Fano manifolds, by a different method. 
What underlies historically behind this theory of cone type 
varieties is the real odd-dimensional 
version of K\"ahler-Einstein manifolds, the structure of Sasaki-Einstein manifolds on their links 
(see the textbook \cite{BoyerGalicki}) as we review in the second section. 
One benefit of our approach is that it also gives alternative proof of the 
algebraic construction of the K-moduli of Fano varieties, without depending on 
the $\delta$-invariant 
nor $L^2$-normalized Donaldson-Futaki invariants, 
but rather we only use 
the normalized volume or the volume density of 
the Einstein metrics, to match more easily with 
theory of log terminal singularities. 
We work over an algebraically closed field $\k$ 
of characteristic $0$. 

\begin{MThm}[{cf., Definitions \ref{def:cone}, 
\ref{def:cs}, \ref{SEsp}, 
Theorem \ref{Kmod.Fcone}}]
For each algebraically closed field $\k$ 
of characteristic $0$, 
$n\in \Z_{>0}$ and $V\in \R_{>0}$, 
there is a proper moduli algebraic $\k$-space of 
$n$-dimensional K-polystable Fano cones over $\k$ 
(a.k.a., Calabi-Yau cones\footnote{
in some literature, it requires slightly more 
constraints for this terminology e.g., \cite{CH}}) 
of the volume density $V$. 

If $\k=\C$, equivalently, 
real $(2n-1)$-dimensional compact 
Sasaki-Einstein spaces of volume density $V$ 
form a compact Moishezon analytic space for each 
fixed $n$ and $V$. 
\end{MThm}

More precise version, logarithmic generalization and 
a corollary are stated and proved as 
Theorem \ref{Kmod.Fcone}, \ref{Kmod.Fcone2} and 
\ref{Kmod.F3} respectively. We conjecture (Conjecture \ref{projconj}) 
that these good moduli spaces are also projective 
schemes with Weil-Petersson type (singular) 
K\"ahler metrics. 

\section{Preparation}\label{sec:prep}

We give a brief review of some preparatory materials. 
Most of (but perhaps not all of) the materials in this section are known 
before as we give references to each.

\subsection{Sasakian geometry}\label{sec:Sasaki}

A classical differential geometric (or contact geometric) 
counterpart of the Ricci-flat K\"ahler cone is the geometry of Sasakian manifolds, 
which we briefly recall. This subsection is 
intended to be more differential geometric review, 
and only works over $\k=\C$ whenever some varieties appear. 

\begin{defn}[Sasakian manifolds cf., \cite{Sasaki, BG}]
A Riemannian manifold $(S,g)$ of (real) odd dimension $2n-1$ with 
$n\in \Z_{>0}$ is said to underlie 
a {\it Sasakian manifold} 
if there is 
a complex structure $J$ on 
$C^o(S):=S\times \R_{>0}$ 
with respect to which $g_{C(S)}$ 
is a K\"ahler metric, 
so that $J$ further extends to the cone 
$C(S):=C^o(S)\cup 0$. 
The corresponding K\"ahler form is, by the definition of 
$g_{C(S)}$, $\omega_{C(S)}:=\sqrt{-1}\partial\overline{\partial} r^2$. 

Then, the corresponding $\xi:=J\bigl(r\frac{\partial}{\partial r}\bigr)$ is called the Reeb vector field. 
To avoid confusion with more algebraic flexible variant 
(positive vector field in Definition \ref{def:cone}), 
we call it the 
{\it metric Reeb vector field}. 
We also have a 
contact $1$-form 
$\eta:=\iota_{\xi}g=\iota_{\frac{\partial}{\partial r}}\omega_{C(S)}$ associated, 
where $\iota$ stands for the interior product. 
The actual Sasaki structure on 
$M$ is in addition to $g$, 
further we encode $\xi$ (or $\eta$ 
equivalently). 
\footnote{although, in most of quasi-regular cases, 
this is redundant by \cite[8.4]{BGN}, \cite[20, 21]{BGKollar}}

There are also 
associated structures such as $(1,1)$ type 
real tensor field 
$\Phi\in \Gamma(M,{\rm End}(T_{S}))$, 
which satisfies 
$\Phi\circ \Phi=-{\rm id}+\xi\otimes \eta$, 
where ${\rm End}(T_{S})$ means the 
endomorphism bundle of the 
(real) tangent bundle, 
${\rm id}$ means the identity map and 
the contact form $\eta$ is defined as 
$\eta(v)=g(\xi,v)$. 
The datum is again 
equivalent to the (almost) complex structure $J$ 
on $C(S)$ by the formula 
$\Phi(v)=J(v-g(\xi,v)\xi)
=\nabla^{\rm LC}_v \xi$ by the Levi-Civita 
connection $\nabla^{\rm LC}$. 
Below, we assume $S$ is compact throughout, 
unless otherwise stated. 
\end{defn}

\begin{defn}[Quasi-regular, irregular]
$\overline{\R \xi}\subset {\rm Isom}(S,g)$ is known to be a 
$r(\xi)$-dimensional compact torus. We denote its complexification as $T(\C)\simeq (\C^*)^{r(\xi)}$. 
If $r(\xi)=1$ (resp., $r(\xi)>1$), we call $(S,g)$ or corresponding 
Fano cone is quasi-regular 
(resp., irregular). 
If 
$(S,g)$ is quasi-regular and further the action of $T(\C)\simeq \C^*$ on $C(S)$ is free, 
we call $(S,g)$ is regular. 
\end{defn}

The following is also well-known (cf., e.g., \cite{BoyerGalicki}). 

\begin{prop}\label{SE.equiv}
For a compact Sasakian manifold $(S,g,\xi)$, 
the following are equivalent: 

\begin{enumerate}
\item \label{SE1}
$(C(S),g_{C(S)})$ is Ricci-flat. 
\item \label{SE2}
$(S,g)$ satisfies $\Ric_g=(\dim(X)-1)g$, hence Einstein manifold in particular 
(called Sasaki-Einstein manifold). 
\item \label{SE3}
In the quasi-regular case, 
$S/(T(\C)\simeq \C^*)$ admits a natural branch divisor with standard coefficients $\Delta$ 
(which is $0$ in the regular case) so that $(S/(T(\C)\simeq \C^*),\Delta)$
is a log K-polystable $\Q$-Fano variety with a conical weak K\"ahler-Einstein metric. 
\end{enumerate}
In the above case, $C(S)\cup 0$ has only log terminal singularity at the vertex. 
\end{prop}

As our highlight, we consider the limiting 
behaviour. 

\begin{defn}
We consider all the (real) $2n-1$-dimensional 
compact Sasaki-Einstein manifolds 
$(S,g)$ of the volume $V$ and denote 
their isomorphic classes 
as $M_{n,V}$. 
\end{defn}

\begin{cor}
    $M_{n,V}$ is precompact 
    with respect to the Gromov-Hausdorff topology. Furthermore, $M_{n,V} (i=1,2,\cdots)$ satisfies the non-collapsing condition (compare \cite{DS}): 
\begin{quote}
For any $0<r\le \diam(S_i,g_i)$, 
there is a positive real number $c$ such that   
we have $\vol(B_r(p,(S_i,g_i))\ge c r^n$. 
Here, $\diam(-)$ means the diameter, 
$\vol$ means the volume and 
$(B_r(p,(S_i,g_i))$ means the 
geodesic ball of radius $r$ and the center $p$ 
in $(S_i,g_i)$. 
\end{quote}    
\end{cor}

\begin{proof}
The former claim follows from the 
Myers theorem and Gromov precompactness 
theorem. The latter claim follows from the 
Bishop-Gromov comparison theorem. 
\end{proof}

From this claim, we naturally expect 
some compactness theorem and the 
main aim of our paper is to give 
a proof to its algebraic analogue. 

Martelli-Sparks-Yau 
\cite{MSY2} studied necessary condition of the existence of Sasaki-Einstein metric 
(using the case study when $(C(S),J)$ in \cite{MSY1}) 
and showed that the volume is an algebraic number. 
\begin{defn}
For a compact Sasaki-Einstein manifold $(S,g,\xi)$, define its volume as 
\begin{align*}
\vol^{\rm DG}(\xi)&:=\frac{1}{(2\pi)^n n!}\int_{C(S)} e^{-r^2}\omega_{C(S)}^n\\ 
	   &=\frac{\vol(S,g)}{\vol{(S^{2n-1}(1))}}. 
\end{align*}
The denominator is with respect to the standard round metric on the unit sphere. 
Note that the above invariant ${\rm vol}^{\rm DG}$ 
is invariant under the rescale of $r$ by $cr (c\in \R_{>0})$. 
\end{defn}

The above volume has a relatively algebraic nature and is determined only by 
a positive vector field 
$\xi$ as follows. 
We leave the details of the proof of the following to 
\cite[\S 6]{CS}, \cite{MSY2} or \cite{Li}.

\begin{prop}[{\cite{MSY2},\cite{CS}, \cite[Proposition  6.6]{CS2}, and Lemma \ref{lem:conv} later}]\label{volvol}

For any compact Sasaki manifold $(S,g,\xi)$, 
the cone $(C(S),J)$ 
is a complex affine variety, which we write as 
$\Spec(R)$, acted by the algebraic $\k$-torus 
$T(\C)$ which induces the $T$-eigenspaces decomposition 
$R=\oplus_{\vec{m}\in {\rm Hom}(T(\C),\C^*)}$ 
and we regard $\xi$ as an element of 
${\rm Hom}(\C^*,T(\C))$. 
Then, we can define and write the 
index character $F$ (\cite{MSY2}) as 
\begin{align*}
F(\xi, t)&:=\sum_{\vec{m}\in M}e^{-t\langle \vec{m}, \xi \rangle}\dim R_{\vec{m}}\\
&=\frac{A_0(\xi)}{t^n}-\frac{A_1(\xi)}{t^{n-1}}+O(t^{2-n})
\end{align*}
and further, if $\xi$ is the metric Reeb vector 
field for the Ricci-flat K\"ahler metric, we have 
\begin{align}
\label{5.1}A_0(\xi)&=\vol^{\rm DG}(\xi)\\ 
\label{5.2}	           &=\widehat{\vol}({\rm val}_{\xi}).
\end{align}
Here \eqref{5.2} is in the sense of C.Li \cite{Li}, which algebraizes the theory and is the topic 
of the next subsection. Actually the 
equality 
$A_0(\xi)=\widehat{\vol}({\rm val}_{\xi})$ 
makes sense and holds true for more general cases 
(abstract Reeb vector field in Definition \ref{def:cone}) as we see in Lemma \ref{lem:conv}. 

In particular the above volume $\vol^{\rm DG}(\xi)$ is determined only by the multi-Hilbert function of 
the decomposition $\Gamma(\O_X)=R=\oplus_{\vec{m}}R_{\vec{m}}$. 
\end{prop}

\begin{Thm}\label{SE.property}
If $(S,g)$ is a compact Sasaki-Einstein manifold, the following hold. 
\begin{enumerate}
\item (\cite{MSY1, MSY2})
$\vol^{\rm DG}(-)$ extends to whole Reeb cone (cf., Definition \ref{def:cone}) and 
is minimized at $\xi$ which is determined by the Sasakian structure. 

\item (\cite{Li, CS2}) \label{gauge.fix}
the metric Reeb vector field $\xi$ of the metric tangent cone $C(S)$ satisfies the 
normalization (or gauge-fixing) condition $A_{C(S)}(\xi)=n$. Here, $A_{C(S)}(-)$ denotes 
the log discrepancy function in the theory of the 
minimal model program and ${\rm val}_\xi$ denotes the naturally corresponding 
valuation of $\O_{C(S),0}$ to $\xi$. 
\end{enumerate}
\end{Thm}

The above results show some algebro-geometric 
natures of the Sasakian geometry. From 
the next section, we mainly work by such 
algebro-geometric framework. 

\subsection{Affine cone} 

Henceforth, we work on more algebro-geometric side, over an algebraically closed field $\k$ of characteristic $0$ 
unless otherwise stated. 
Take an arbitrary algebraic torus $T$, 
we set 
$N:={\rm Hom}(\G_m, T)$, $M:=N^{\vee}={\rm Hom}(T,\G_m)$ throughout. 
Before going further, we promise the further toric notation throughout the paper 
(which follows \cite{Od24b}). 
\begin{ntt}\label{ntt}
For any rational polyhedral cone $\tau \subset N\otimes \R$, we 
denote an affine toric variety corresponding to it as 
$U_{\tau}(\supset T)$ and its $T$-invariant vertex (closed point) as $p_\tau$. 
We denote the quotient stack $[U_\tau/T]$ as 
$\Theta_\tau$. 
We often consider non-zero irrational element $\xi\in (\tau\setminus N_{\Q}:=N\otimes_{\Z} \Q)$ 
i.e., of $\Q$-rank $r$. 

On the dual side, we set $M:={\rm Hom}(N,\Z)$ as the dual lattice, 
$\tau^{\vee}:=\{x\in M\otimes \R\mid \langle \tau, x\rangle \subset \R_{\ge 0}\}$, 
$\mathcal{S}_\tau:=\tau^{\vee}\cap M$. If we regard $M$ as the 
the character lattice of $T$, then the character of $T$ which corresponds to 
$\vec{m}\in M$ is denoted by $\chi_{\vec{m}}$. 
\end{ntt}

\begin{defn}[cf., \cite{LS13, CS, CS2}]\label{def:cone}
\begin{enumerate}
\item \label{def1}
Consider a normal affine variety $X$ with the vertex $x$ and an action of an algebraic torus $T$, 
which is {\it good} in the sense of \cite{LS13}, 
i.e., the action is effective and $x$ is contained in the closure of any $T$-orbit 
which characterizes $x$. 
\footnote{When $r=1$, $T\curvearrowright X\ni x$ is also sometimes said to be a quasicone in the literature. }
We consider the decomposition 
$$R:=\Gamma(X,\O_X)=\oplus_{\vec{m}\in M}R_{\vec{m}}$$ 
and define the {\it moment monoid} as 
\begin{align*}
\sigma_R:=\{\vec{m}\in M\mid R_{\vec{m}}\neq 0\}, 
\end{align*}
the {\it moment cone} as 
\begin{align*}
\R_{\ge 0}\sigma_R \subset M\otimes \R,
\end{align*}
and the {\it Reeb cone} as 
\begin{align*}
C_R:=\{\xi\in N\otimes \R\mid \langle \vec{m}, \xi\rangle >0 \quad 
\forall \vec{m}\in \sigma_R \setminus \{0\}\}. 
\end{align*}
We call an element $\xi$ of $C_R$ an 
{\it abstract Reeb vector field} or 
{\it positive vector field}. 
An important observation is that, for each $\xi$, we can associate a 
valuation ${\rm val}_{\xi}$ with the center $x\in X$ (Definition \ref{valxi}). 
See also there is an interpretation of $C_R$ as an analogue to 
the K\"ahler cone (cf., e.g., \cite[\S 2]{BvC}), where it is called the {\it Sasaki cone}. 

\item 
An {\it affine cone}
\footnote{or {\it generalized affine cone}, 
to clarify that it is in the broader sense than 
the classical case i.e., when $r=1$ and 
$\xi$ is regular 
as \cite[(3.8)]{Kol13}} 
is simply a triplet $(X, T\curvearrowright  X, \xi)$ 
where 
$X$ is an affine algebraic $\k$-scheme, 
$T\curvearrowright  X$ is a good action, and $\xi$ 
is its abstract Reeb vector field. 
We denote the dimension of $X$ as $n$. 

$\xi$ is called regular (resp., quasi-regular) if 
$r=1$ and $T$ acts freely on $X\setminus\{x\}$ 
(resp., $r=1$). In the case when $\xi$ is regular, 
we denote the corresponding polarized projective 
scheme as $(V,L)$ so that $X=\Spec \oplus_{m\ge 
0}H^0(V,L^{\otimes m})$. If $\xi$ is quasi-regular, 
i.e., in the general $r(\xi)=1$ case, 
we can similarly construct the quotient 
$[(X\setminus x)/T]\to (X\setminus x)//T=:V$ 
as a projective scheme (the GIT quotient). 
If $X$ is normal, $V$ is also automatically normal 
(cf., \cite[Chapter I]{GIT}) and we can consider the 
ramification index $m(D)$ for each prime divisor 
$D$ on $V$ and form the {\it standard ramification 
divisor} as a Weil $\Q$-divisor 
$\sum_D \frac{m(D)-1}{m(D)}D$ on $V$. 
This morphism  
$(X\setminus x)\to (V,\sum_D \frac{m(D)-1}{m(D)}D)$ 
is called the {\it (algebraic)
Seifert $\G_m$-bundle} 
in \cite{Kol04, Kol13} (cf., also \cite[\S 3.1]{LL}), 
It globally realizes the so-called 
transverse K\"ahler structure on the 
locally orthogonal direction 
to the Reeb foliation in $r=1$ case 
(in the irregular case, 
one can only locally realize it cf., 
\cite{BoyerGalicki}). 

\item 
A {\it $\Q$-Gorenstein (affine) cone}\footnote{temporary name in this paper} 
is (generalized) affine cone where 
$X$ is reduced which is normal crossing in 
codimension $1$, satisfying the Serre condition $S_2$, $K_X$ is $\Q$-Cartier, 
$T\curvearrowright  X$ is a good action, and $\xi$ is its abstract Reeb vector field. 
It is called a 
{\it Fano cone} if $X$ is also log terminal. 
We also call it a {\it (kawamata-)log-terminal cone} 
in this paper. Similarly, if 
$X$ is log canonical (resp., semi-log-canonical), we call it 
a {\it log canonical cone} (resp., {\it semi-log-canonical cone}). 
\end{enumerate}
\end{defn}

There is a valuative interpretation of positive 
vector field. 

\begin{defn}[Monomial valuations {cf., \cite{BFJs}, 
\cite[\S 2.2]{Li}}]\label{valxi}
If $x\in X\curvearrowleft T$ is an irreducible 
affine cone, then 
for each positive vector field $\xi\in C_R$ (see 
Definition \ref{def:cone}), we can associate a 
valuation ${\rm val}_\xi$ of 
$X$ centered at the vertex $x$ 
as follows: 
\begin{align}
{\rm val}_\xi(f):=
\min_{\vec{m}\in \sigma_R}\{\langle \vec{m},\xi  \rangle \mid f_{\vec{m}}\neq 0\},
\end{align}
where $f=\sum f_{\vec{m}}$ 
is the decomposition for $R=\oplus R_{\vec{m}}$. 
If we take log resolution of $(X,\sum_i {\rm div}(f_i))$ 
where $f_i$ are $T$-homogeneous generators of $R$, 
it follows that the above ${\it val}_\xi$ 
is quasi-monomial (compare \cite{BFJs}). 
\end{defn}

The following should be known to experts. 

\begin{lem}\label{lem:cK}
\begin{enumerate}
\item \label{Ktriv}
In the setup of Definition \ref{def:cone} \eqref{def1}, 
if $mK_X$ is Cartier for some positive integer $m$, 
then it is automatically linearly trivial. 

\item \label{Ktriv.Fano}
For a $\Q$-Gorenstein affine cone $T\curvearrowright  X\ni x$ with the abstract Reeb vector field $\xi$, 
there is a positive integer $l$ and a nowhere vanishing holomorphic section 
$\Omega\in \Gamma(\mathcal{O}(l K_{X}))$ and some real number $\lambda$ so that 
$$L_{\xi}\Omega=\sqrt{-1}\lambda \Omega.$$
Here, $L_{\xi}$ means the Lie derivative. Moreover, 
$T\curvearrowright  X\ni x$ with $\xi$ is a Fano cone if and only if 
$\lambda>0$. 
\end{enumerate}
\end{lem}

\begin{proof}
For \eqref{Ktriv}, 
by taking the index $1$-cover with respect to 
$K_X$, we can reduce to the quasi-Gorenstein case. 
By the complete reducibility of 
$T\curvearrowright \Gamma(mK_X)$ 
we can write $K_X$ as a 
$T$-invariant divisor which does not contain the vertex, but then 
it is easy to see the support is empty. 
See \cite{PS, LS} for discussions in more general case. The classical case i.e., 
for the regular action of $T$ with $r=1$, see 
\cite[3.14]{Kol13} for the proof of this 
\eqref{Ktriv}. 
For \eqref{Ktriv.Fano} is proved in \cite[Lemma 6.1, 6.2]{CS2} for 
the normal case. For non-normal case, the same 
proof works by combining with \cite[16.45, 16.47 and the proof]{GZ}. 
\end{proof}

The last statement also follows from 
the Seifert $\G_m$-bundle interpretation 
by \cite[\S 9.3]{Kol13}, \cite[\S 3.3.1]{LL} 
(cf., also \cite[\S 3.1]{LLX}). 
$\lambda$ of the above is interestingly 
interpretted as log discrepancy of $X$ by 
C.Li \cite{Li}. For instance, its positivity was known to be equivalent to 
log terminality of $X$ as observed in \cite[\S 6]{CS2}.

\begin{defn}\label{def:Hilbfun}
\begin{enumerate}
\item The {\it multi-Hilbert function} of an affine cone $T \curvearrowright X\ni x$ 
with the positive 
vector field $\xi \in N_\R$ 
is a map $M\to \mathbb{Z}_{\ge 0}$ defined by 
$\vec{m}\mapsto \dim R_{\vec{m}}=:\chi_X(\vec{m})$. 

\item (cf., \cite{MSY2, CS, CS2})
The {\it multi-Hilbert series} of an affine cone $T \curvearrowright X\ni x$ with the positive 
vector field $\xi \in N_\R$ 
\footnote{Originally called {\it index character} by \cite{MSY2, CS, CS2} in the context of 
equivariant index theory but this term may sound more familiar to algebraic geometers}
is defined as 
\begin{align*}
F(\xi, t)=\sum_{\vec{m}\in M}e^{-t\langle \vec{m}, \xi \rangle}\dim R_{\vec{m}}. 
\end{align*}
A priori one can regard it only 
as a formal function but 
we review in Lemma \ref{lem:conv} 
(from \cite{MSY1, CS}) later that 
as far as $R$ is 
finitely generated, the series is meromorphic 
around $0$ so that it encodes the multi-Hilbert function $\{\dim R_{\vec{m}} \}
_{\vec{m}}$ as the Fourier coefficients. 
Obviously, the moment monoid, the moment cone and the Reeb cone 
are determined only by the multi-Hilbert function. 

\item (cf., \cite{CS}) 
Take an affine cone $\tilde{T} \curvearrowright X\ni x$ with an algebraic $\k$-torus 
$\tilde{T}(\supset T)$, its character lattice $\tilde{M}:={\rm Hom}(T,\G_m)$ 
(resp.,  $\tilde{N}:={\rm Hom}(\G_m,T)$) 
\footnote{This notation is set in this way as later we often apply to the central fiber of 
the test configuration (Definition \ref{def:cs} \eqref{def:tc}) so that $\tilde{T}$ is often $T\times \G_m$, 
unless it is a product test configuration.} 
and $\eta \in \tilde{N}\otimes \R$. We denote the decomposition for the $\tilde{T}$-action of 
$R:=\Gamma(\O_X)$ as $\oplus_{\tilde{m}} R_{\tilde{m}}$. 

The {\it weighted (multi-)Hilbert series}\footnote{called {\it weighted character} in \cite{CS}} 
of the affine cone $\tilde{T} \curvearrowright X\ni x$ 
with the positive 
vector field $\xi \in N_\R$, 
for the $\eta$-direction,  
is defined as 
$$
C_\eta(\xi,t):=\sum_{\tilde{m}\in \tilde{M}}e^{-t\langle \vec{m}, \xi \rangle}\langle \tilde{m}, \eta \rangle 
\dim R_{\tilde{m}}. 
$$

\item (cf., \cite{HS}) \label{def:Tfppf}
Consider any faithfully flat affine family between 
algebraic $\k$-schemes 
$\pi\colon \X=\Spec_{S}(\mathcal{R}) \to S$ with connected $S$, with $\mathcal{O}_S$-algebra 
$\mathcal{R}$, 
where an algebraic $\k$-torus acts on $\X$ fiberwise. 
We apply the complete reducibility of $T$ to $\mathcal{R}$ to 
obtain the decomposition $\mathcal{R}=\oplus_{\vec{m}\in M}\mathcal{R}_{\vec{m}}$ 
where $\mathcal{R}_{\vec{m}}$ denotes the $\mathcal{O}_S$-module 
corresponding to the character $\vec{m}$. 
If $\mathcal{R}_{\vec{m}}$ for any $\vec{m}$ is faithfully flat over $S$, 
we call the family $\pi\colon \X\to S$ is {\it $T$-equivariantly faithfully flat over $S$}, 
or $T$-fppf for short (if no confusion), in this paper. 
Note that this condition is said to be admissibility of the deformation 
in \cite{HS} and is 
stronger than the faithful flatness of $\pi$. See also the following Lemma \ref{lem:T} \eqref{lem:T2}. 
\end{enumerate}
\end{defn}

Here is some useful general lemma. 

\begin{lem}\label{lem:T}
\begin{enumerate}
    \item (Characterization of good action) \label{lem:T1}
Take an affine algebraic $\k$-scheme $X=\Spec(R)$ on which an algebraic 
$\k$-torus $T$ acts. It is a good action if and only if the 
multi-Hilbert functions are all finite i.e., $\dim(R_{\vec{m}})<\infty$ 
for any $\vec{m}\in M$, 
and the moment cone 
$\R_{\ge 0}\sigma_R$ 
is strictly convex i.e., does not contain 
any line. It is further equivalent to the 
non-triviality of the Reeb cone 
$C_R\neq \emptyset$. 

   \item (Constancy of multi-Hilbert function) 
   \label{lem:T2}
As Definition \ref{def:Hilbfun} \eqref{def:Tfppf}, 
consider any flat affine family (resp., $T$-equivariantly faithfully flat) between 
algebraic $\k$-schemes 
$\pi\colon \X\to S$ with irreducible $S$ whose generic point is $\eta$, 
where an algebraic $\k$-torus acts on $\X$ fiberwise. 
If we compare the multi-Hilbert function of the generic fiber $X_\eta$ 
and any fiber $X_s$ for $s\in S$, we have that 
$\chi_{X_s}(\vec{m})$ is either $0$ or $\chi_{X_\eta}(\vec{m})$. 
Furthermore, 
$\chi_{X_s}(\vec{m})=\chi_{X_\eta}(\vec{m})$ for any $\vec{m}$ 
if and only if $\pi$ is $T$-equivariantly faithfully flat. 

\item 
\label{lem:T3}
In the $T$-equivariantly faithfully flat situation of the above \eqref{lem:T2}, 
if $T\curvearrowright \X_s$ for some $s\in S$ is a good action, 
it holds for any $s\in S$. 
\end{enumerate}

\end{lem}
\begin{proof}
We first prove \eqref{lem:T1}. 
The only if direction is discussed in \cite[\S 3.1]{LLX} 
(cf., also \cite{LS}). For the converse i.e., 
the if direction, note that 
$R_{\vec{0}}=\k$ since it is a finite extension of 
$\k$ and is an integral domain, while we assume 
$\k$ is algebraically closed. Since 
$R_{\vec{0}}=R^T$ i.e., coincides with the 
$T$-invariant subring of $R$, it follows that 
the polystable locus $X^{ps}$ of $X$ consists of finite $T$-orbits. 
Furthermore, by the standard arguments using 
the Raynold operator, it follows that $X^{ps}$ is  
connected. Hence it consists of 
a single $T$-orbit, say $Tx$ for some closed point 
$x\in X$. Denote the identity component of 
the stabilizer ${\rm stab}(x)$ of $x$ as 
${\rm stab}^0(x)$. Then the 
character lattice of $T/{\rm stab}^0(x)$ 
should be trivial since otherwise it would contradicts with 
the strict convexity of $\R_{\ge 0}\sigma_R$. 

Next we prove \eqref{lem:T2}. 
We can and do assume $S$ is affine and write the family as 
$\X={\rm Spec}_R(\mathcal{R})$ by a 
$(R:=)\Gamma(\mathcal{O}_S)$-algebra. 
Apply the complete reducibility of $T$ to 
$\mathcal{R}$ to obtain 
$\mathcal{R}=\oplus_{\vec{m}\in M}\mathcal{R}_{\vec{m}}$, 
with $R$-modules  $\mathcal{R}_{\vec{m}}$. 
Since $\mathcal{R}$ is $R$-flat, 
each $\mathcal{R}_{\vec{m}}$ is also $R$-flat hence 
locally free module because of the 
finitely generatedness assumption. Therefore, the 
assertion follows. 
Finally, \eqref{lem:T3} follows from 
\eqref{lem:T1} and \eqref{lem:T2}. 
\end{proof}

We note that a certain embedded version of 
\eqref{lem:T1} (resp., \eqref{lem:T2}) is 
partially 
proved in literature as \cite[4.5]{CS}, 
\cite[4.1.19]{KR} 
(resp., \cite[after Definition  5.1]{CS}). 
We also review the following for the next section. 

\begin{lem}[\cite{CS, CS2}]\label{lem:conv}
Recall that we set $n:=\dim(X)$. 
\begin{enumerate}
\item \label{charA}
The multi-Hilbert series 
$F(\xi, t)$ can be written as 
$$F(\xi, t)=\dfrac{A_0(\xi)}{t^n}-\dfrac{A_1(\xi)}{2 t^{n-1}}+O(t^{2-n}),$$ 
with $t$ around $0\in \C$ where $A_i(\xi) (i=0,1)$ are $C^\infty$-function on the Reeb cone $C_R$ 
(Definition \ref{def:cone}). 
In general, we have 
\begin{align}\label{lem:conv:eq}
A_0(\xi)=\widehat{\vol}({\rm val}_{\xi}), 
\end{align}
where the latter 
$\widehat{\vol}(-)$ 
means the normalized volume in the sense of \cite{Li} 
(Definition \ref{def:nv}). Note that this holds for 
general affine cone as proved in \cite[\S 4]{CS}. 

Furthermore, if $r(\xi)=1$, $T$ is regular and 
$\xi$ is a generator of $N$, 
in the notation of Definition \ref{def:cone}, 
$A_0(\xi)=(L^{n-1})$, $A_1(\xi)=(L^{n-2}.K_V)$. 
Note the obvious homogeneity 
$A_i(c\xi)=c^{-n+i}A_i(\xi)$ for $i=0,1$. 

\item \label{charB}
The weighted multi-Hilbert series 
$C_\eta(\xi,t)$ 
can be written as 
$$C_\eta(\xi,t)=\dfrac{B_0(\xi)}{t^{n+1}}-\dfrac{B_1(\xi)}
{2t^{n}}+O(t^{1-n}),$$ 
where $B_i(\xi) (i=0,1)$ are $C^\infty$-function on 
the Reeb cone $C_R$. 

B functions are directional derivative of A functions in the sense that 
$B_i(\xi)=-D_{\eta}A_i(\xi)$ for $i=0,1$ 
where $D_\eta$ denotes the 
directional derivative along the direction $\eta$. 
Hence, we again have 
the homogeneity 
$B_i(c\xi)=c^{-n-1+i}A_i(\xi)$ for $i=0,1$. 

\end{enumerate}
\end{lem}

\begin{proof}
See \cite[Theorem 3]{CS} 
(which depends on \cite[\S 5.8]{KR}) 
for the proof of \eqref{charA}. 
\eqref{charB} follows from the 
derivation by terms of 
$\sum_{\tilde{m}\in 
\tilde{M}}e^{-t\langle \tilde{m},\xi-s\eta\rangle}
\dim R_{\tilde{m}}$. See the details in the proof of 
\cite[Theorem 4]{CS}. 
\end{proof}

\subsection{K-stability of affine cones}

\begin{defn}\label{def:cs}
\begin{enumerate}
\item \label{def:tc} (cf., \cite{CS, CS2}) 
A (affine $T$-equivariant) 
{\it test configuration} of  
an affine cone $(T\curvearrowright X, \xi)$ means 
$T\times \G_m$-equivariant affine $T$-equivariantly faithfully
\footnote{this $T$-equivariant faithful flatness is necessary e.g. to 
avoid $X\times (\A^1\setminus \{0\})$ and some other pathological examples, which is missed in some literature} 
flat morphism $\pi\colon \X \to \A^1$ 
from a normal affine scheme $\X$, 
where 
$T$ acts only fiberwise (while acting trivially on the base $\A^1$) 
and $\G_m$ acts multiplicatively on the base $\A^1$ such that its 
restriction to $\pi^{-1}(\A^1\setminus \{0\})$ is the product $X\times (\A^1\setminus \{0\})$. 
A test configuration is called product test configuration if 
there is a $T$-equivariant isomorphism $\X\simeq X\times \A^1$. For simplicity, we sometimes 
abbreviate the whole data of a test configuration 
just as $\X$ if it does not make confusion. 

\item \label{def:df} (cf., \cite{CS})
The {\it Donaldson-Futaki 
\footnote{defined and coined the name in \cite{CS}, 
after \cite{Don02} which generalizes 
\cite{Futaki}} 
invariant} 
${\rm DF}(\X,\xi)$ of a test configuration 
$(T\times \G_m)\curvearrowright \X\xrightarrow{\pi} 
\A^1$ is defined (up to a dimensional constant $2((n+1)!(n-1)!)$) as 
\begin{align}
{\rm DF}(\X,\xi):=(n+1)A_1(\xi)B_0(\xi)-
nA_0(\xi)B_1(\xi), 
\end{align}
where 
$A_i(\xi)$ are that of any fibers 
$\pi^{-1}(s)$ for $s\in \A^1$ (see 
Lemma \ref{lem:T} \eqref{lem:T2} 
for the independence on $s$) 
$B_i(\xi)$ are that of 
$\X_0=\pi^{-1}(0)$ defined in Definition 
\ref{def:Hilbfun} 
\eqref{charB}. 

If $r=1$, $T$ is regular and $\xi$ is 
the generator of $N$, 
in the notation of Definition \ref{def:cone}, 
we have 
\begin{align}
 \qquad \qquad&A_0(\xi)=(L^{n-1}),&A_1(\xi)=(L^{n-2}.K_V), \\
 \qquad \qquad&B_0(\xi)=(\mathcal{L}^{n}),&B_1(\xi)=(\mathcal{L}^
{n-1}.K_{\mathcal{V}/\P^1}), 
\end{align}
where $(\mathcal{V},\mathcal{L})$ 
denotes the (polarized projective) 
test configuration in the sense of 
\cite{Don02} which arise as 
$((\X\setminus \overline{\G_m\cdot (x,1)})/\G_m)\to 
\A^1$ by \cite{Wang, OdDF} 
so that it fits to the original intersection 
number formula in {\it loc.cit} for the 
polarized projective setup. 

Then, the (generalized) 
affine cone $(T\curvearrowright X, \xi)$ is 
{\it K-stable (resp., K-semistable\footnote{Notes added: see also 
\cite{Wu, LW} for  
rephrasing in more non-archimedean terminology. 
The latter appeared after our paper.})} 
if and only if 
$\DF(\X,\xi)>0$ (resp., $\DF(\X,\xi)\ge 0$) 
for any non-trivial affine test configuration 
$\pi\colon \X \to \A^1$. It is said to be 
{\it K-polystable} if 
it is K-semistable and $\DF(\X,\xi)=0$ only occurs when 
$\X$ is a product test configuration. 

\item (cf., \cite{LX14}) 
A test configuration of a Fano cone is called a 
{\it special test configuration} 
if $(\X,\X_0)$ is purely log terminal. 

\item (cf., \cite{CS, CS2}) 
If we allow ourselves to use the normalized volume \cite{Li} 
to be reviewed in \S \ref{sec:normvol}, 
for a special test configuration $\X$ 
of a Fano cone, 
its {\it Donaldson-Futaki invariant} is 
equivalently defined as 
\begin{align*}
\DF(\X,\xi)=\frac{d}{dt}|_{t=0} \widehat{\vol}_{\X_0}({\rm val}_{\xi-t\eta}), 
\end{align*}
up to a dimensional constant 
$2((n+1)!(n-1)!)$. 
where $\eta$ is the holomorphic vector field induced by the $\G_m$-action on the central fiber $\X_0$ of $\X$. 
Note also that if we normalize $\xi$ and 
$\eta$ i.e., to multiply suitable positive real 
number constants, 
so that $\xi-\epsilon \eta$ 
for $\epsilon\ll 1$ all satisfy 
the gauge fixing condition (\cite[Definition 6.3, 
Proposition 6.4] {CS2}) 
i.e., replace $\eta$ by 
$T_{\xi}(\eta):=\frac{A(\xi)\eta-A(\eta)\xi}{n}$ 
(cf., \cite[3.9]{LLX}), 
we can replace the above $\vvol$ by 
the (unnormalized) volume function $\vol(-)$ 
i.e., consider instead 
$$D_{-T_{\xi}(\eta)}\vol_{\X_0}(\xi)$$ 
as it has the same sign as ${\rm DF}(\X)$ 
where $D_{-T_{\xi}}(-)$ again denotes 
the directional derivative. 
One of its benefits is that we know the 
{\it strict convexity} of 
$\vol(-)$ on $C_R$ by \cite{MSY2} 
(later generalized by \cite[\S 3.2.2]{LX} 
to possibly singular T-varieties 
using the Okounkov body \cite{Okounkov, LM09}). 
\end{enumerate}
\end{defn}

\begin{ex}[From one parameter subgroup to test configuration]
If we consider a multi-Hilbert scheme $H$ for affine closed subscheme of $\A^N$ 
of positive weights for an algebraic $\k$-torus $T$ with fixed multi-Hilbert function, 
the centralizer of $T$ in ${\rm GL}(N)$ which we denote as $G$, naturally acts on 
$H$. If we take one parameter subgroup $\rho\colon \G_m\to G$ and a point $[X]\in H$, 
the family over $\overline{\rho(\G_m)\cdot [X]}$ is automatically a test configuration. 
This is essentially the way \cite[Definition 5.1]{CS} first defined the test configurations. 
Indeed, it is also easy to see that all test configuration, in our above more abstract sense, 
arises in this manner. 
\end{ex}

Note that more Donaldson-type (i.e., \cite{Don02}) definition 
is also proved to be equivalent, as an analogue of \cite{LX14}, by Wu \cite{Wu.thesis}. 
There is also a partial related result in 
\cite[4.3]{LWX}, which we extend in the proof of 
Theorem \ref{A1}. 

We also prepare the Duistermaat-Heckman measures and 
norm functionals analogous to the 
global projective setup, following \cite{DH, Od12a, 
Hisamoto, BHJ, Dervan, Wu}. 

\begin{defn}
\begin{enumerate}
    \item (\cite{DH, Hisamoto})
For a $\Q$-Gorenstein affine cone $T\curvearrowright X\ni x$ with 
an additional commuting action of $\G_m$, we define the 
{\it Duistermaat-Heckman measure} $DH(X)$ for the 
$\G_m$-action as 
the probability measure 
$$\lim_{c\to \infty}
\biggl(\sum_{\substack{\lambda\in \Z, \\ 
\vec{m}\in M, \langle 
\vec{m},\xi\rangle <c}}\dfrac{\dim(R_{\vec{m}})_\lambda}
{\dim(R_{\vec{m}})}\quad \delta_{\frac{\lambda}{\langle m,\xi 
\rangle}}\biggr).$$ 
Here, $(R_{\vec{m}})_\lambda$ is the 
$\k$-linear subspace of $R_{\vec{m}}$ 
with the $\G_m$-weight $\lambda$ and 
$\delta_a$ of $a\in \R$ denotes the 
Dirac measure supported on $a\in \R$. 
Note that the existence of the limit measure 
(convergence) easily follows from 
approximating $\xi$ by rational vectors, 
to which \cite{Hisamoto, BHJ} applies. 

The Duistermaat-Heckman measure of a ($T$-
equivariant affine) 
test configuration $\pi\colon \X\to \A^1$ refers 
to that of 
$\G_m\curvearrowright \X_0=\pi^{-1}(0)$. 
\item (\cite[Chapter V]{Wu} cf., also \cite{BHJ, Dervan}) 
For a ($T$-equivariant affine) test 
configuration $(T\times \G_m)\curvearrowright \X\xrightarrow{\pi} \A^1$ 
of the Fano cone $T\curvearrowright X$, 
we define 
$(I^{\rm NA}-J^{\rm NA})(\X)$ as 
$I^{\rm NA}(\varphi_\X)-J^{\rm NA}(\varphi_\X)$ 
of \cite[Definition V.8]{Wu}, 
where $\varphi_\X$ is the Fubini-Study function 
(\cite[IV.2.1]{Wu}, compare 
\cite{BJ})
induced by $\X$ (or $\pi_* \O_\X$). 
\end{enumerate}    
\end{defn}

\begin{lem}[{cf., \cite[2.7(i), p.2283 l2]{Od12a}, \cite[\S 7.2]{BHJ} \cite[3.10, 3.11, \S 4]{Dervan}}]
The strict positivity $(I^{\rm NA}-J^{\rm NA})(\X)>0$ holds unless 
$\X$ is the trivial test configuration i.e., 
$(T\times \G_m)$-equivariantly isomorphic to 
$X\times \A^1$. 
\end{lem}

\begin{proof}
By approximating $\xi$ by a sequence of 
rational elements in $N_\Q$, due to \cite[7.8]{BHJ}, 
it is enough to show that 
$J^{\rm NA})(\X)>0$ unless 
$\X$ is the trivial test configuration. 
Furthermore, again by the same approximation, 
\cite[7.8]{BHJ} implies that 
$J^{\rm NA}(\varphi_\X)=
\sup {\rm supp}DH(\X)-\int_{\R}\lambda 
DH_{\X}$. If this attains $0$, 
it follows that $DH_\X$ is a Dirac measure 
whose support is a single point. 
Since we assume $\X$ is normal, 
it is then trivial test configuration. 
See also \cite[4.7]{Dervan}. 
\end{proof}

\begin{Prob}[cf., \cite{Od13}]\label{thm:lt}
Can we prove that any K-semistable $\Q$-Gorenstein affine cone $(T\curvearrowright X, \xi)$ is semi-log-canonical 
by extending the method to \cite{Od13}? What is the 
differential geometric meaning behind it if it is true? 
(cf., e.g. \cite{BG}). 
\end{Prob}

The following is the Yau-Tian-Donaldson type correspondence, which generalizes 
the case of Fano manifolds mainly after 
\cite{BermanKps, DatarSzekelyhidi, 
CDS, Tian15}. 

\begin{Thm}\label{CollinsSzekelyhidi.theorem}
\begin{enumerate}
    \item (\cite{CS, CS2})
A (smooth) 
Fano cone $T\curvearrowright X$ admits a structure of Ricci-flat K\"ahler metric, 
which is a cone metric of some Sasaki metric on the link, with 
metric Reeb vector field $\xi$ if and only if 
$(T\curvearrowright X\ni x, \xi)$ is K-polystable in the sense of Definition \ref{def:cs}. 
 \item (\cite{Li21}) \label{sing.CS}
 A (possibly klt) 
Fano cone $T\curvearrowright X$ admits a structure of 
Ricci-flat (weak) K\"ahler metric, 
which is a cone metric of some Sasaki metric on the 
smooth locus of the link, with 
metric Reeb vector field $\xi$ if and only if 
$(T\curvearrowright X\ni x, \xi)$ is K-polystable in the sense of Definition \ref{def:cs}.

\end{enumerate}
\end{Thm}
Such a cone type complete Ricci-flat K\"ahler metric $g$ 
on 
$(T\curvearrowright X\ni x, J, \xi)$
is often called 
{\it Ricci-flat K\"ahler cone metric} and such a 
Fano cone with Ricci-flat K\"ahler cone metric 
(cf., Prop \ref{SE.equiv}) 
is often called {\it Calabi-Yau cone} in the 
literatures under some additional 
smoothness assumptions, 
which is in particular Fano cones. 
We call $\vvol(x\in X)$ its {\it volume density} 
(cf., Proposition \ref{volvol} for the motivation). 

We need to be careful not to 
mix up with log canonical cone we define 
in Definition \ref{def:cone}, 
which are often the affine cone over 
polarized projective (log canonical) Calabi-Yau 
varieties (see \cite[3.1]{Kol13}). 
We also note the latter singular version 
\eqref{sing.CS} relies on the recent 
theory of ``weighted" framework by using the moment maps (cf., \cite{Inoue, Lahdili, Inoue2, HanLis, AJL, Li21}). 

Motivated by the above Definition 
\ref{CollinsSzekelyhidi.theorem} 
\eqref{sing.CS}, 
we introduce the following terminology, 
allowing singularities to extend the notions 
reviewed in \S \ref{sec:Sasaki}. 

\begin{defn}\label{SEsp}
A (compact) 
{\it Sasaki-Einstein space} 
(or {\it Sasaki-Einstein manifold with singularities}) 
is a metric 
space which is decomposed as 
$S=S^{\rm reg}\sqcup S^{\rm sing}$ 
(an open dense subset which is a manifold and closed singular locus) 
with an Einstein (Riemannian) metric $g$ on 
$S^{\rm reg}$ which induces the metric 
structure, together with 
a Reeb vector field 
$\xi \in \Gamma(T_{S^{\rm reg}})$, 
such that the following holds: 
\begin{quote}
corresponding complex structure 
extends to the 
cone $C(S):=(S\times \R_{>0})\cup 0$ 
which underlies a 
Ricci-flat weak K\"ahler cone 
(in the sense of \cite{Li} and 
the previous 
Theorem \ref{CollinsSzekelyhidi.theorem} 
\eqref{sing.CS}). 
\end{quote}
\end{defn}

We hope for more intrinsic equivalent 
definition. 

\subsubsection{CM $\R$-line bundle}

The CM line bundle, first introduced by 
\cite[\S 10, 10.3-5]{FS90} and 
generalized by \cite{PT06, FR06} 
to singular setup, gives a canonical ample line bundle on 
the K-moduli of canonically polarized varieties (\cite{Fjns, KP, PX}), 
polarized Calabi-Yau varieties (\cite{Vie,Od21}), 
Fano varieties (\cite{CP,XZ20}) as 
we expect the same for more general polarized 
K-stable polarized varieties (cf., \cite{FS90, Od10}). 

As in the setup for flat families of 
the polarized varieties, we first 
introduce CM $\R$-line bundle 
on the moduli stack and then discuss its 
descended $\R$-line bundle on the good 
moduli space. Note that it is originally motivated by the generalization 
theory of 
Weil-Petersson metrics \cite{FS90}, 
which was originally obtained as a Quillen-type metric for 
virtual vector bundles \`{a} la Donaldson. 

\begin{defn}\label{CMline}
Consider an arbitrary flat family of 
$n$-dimensional Fano cones 
$T\curvearrowright \X\to S$, 
whose relative coordinate ring 
we decompose as 
$\X={\rm Spec}_S\oplus_{\vec{m}}
\mathcal{R}_{\vec{m}}$. 
Here, $\vec{m}$ runs over the character lattice 
$M:={\rm Hom}(T,\G_m)$. 
Then, we consider the following series 
\begin{align}
L(\xi,t):=\bigotimes_{\vec{m}\in M} (\det \mathcal{R}_{\vec{m}})^{\otimes e^{-t \langle \vec{m},\xi\rangle}}, 
\end{align}
which give $\R$-line bundle i.e., 
element of ${\rm Pic}(S)\otimes_{\Z}\R$ for each complex number $t$ 
with ${\rm Re}(t)>0$ (the same proof as \cite[4.2]{CS}). 
Then, by the same arguments as \cite[4.3]{CS}, 
it has following expansion of meromorphic type: 
\begin{align}
L(\xi,t)=\frac{\mathcal{B}_0 n!}{t^{n+1}}+
\frac{\mathcal{B}_1 (n-1)!}{t^{n}}+O(t^{1-n}), 
\end{align}
if we write in the additive notation (note that 
${\rm Pic}(S)\otimes \R$ is a finite dimensional 
real vector space). 
After the following normalization
\begin{align}
\mathcal{C}_0(\xi)&:=n!\mathcal{B}_0, \\ 
\mathcal{C}_1(\xi)&:=-\frac{(n-1)\cdot (n!)}{2}
\mathcal{B}_0-(n-1)!\mathcal{B}_1, 
\end{align}
we define the {\it CM $\R$-line bundle} 
$\lambda_{\rm CM}(T\curvearrowright \X)$ 
as $-(n-1) A_1(\xi)\mathcal{C}_0+n A_0(\xi)\mathcal{C}_1$ as 
an element of ${\rm Pic}(S)\otimes \R$. 
Note that the above normalization is suitable in the sense that 
if $\xi$ is regular and $S$ is a smooth proper curve, 
\begin{align}
\deg \mathcal{C}_0(\xi)&=(\mathcal{L}^{n}), \\ 
\deg \mathcal{C}_1(\xi)&=(\mathcal{L}^{n-1}.K_{\Y/S}), 
\end{align}
where $\Y\to S$ denotes the 
$\xi$ quotient of $\X\setminus \sigma(S)$ 
where $\sigma$ denote the vertex section, 
and $\mathcal{L}$ is the corresponding ample line bundle. 
\end{defn}

See also a related general construction in \cite{Inoue3}. 


\subsection{Donaldson-Sun theory and 
normalized volume}
\label{sec:DSreview}

\subsubsection{Original work of Donaldson-Sun} 

In \cite{DSII}, 
for log terminal singularities with 
(singular) K\"ahler-Einstein metrics, 
the analytic and even 
algebraic nature of the metric tangent cone is 
started to explore. This is 
somewhat a local analogue of \cite{DS}. 
As a technical assumption,  
\cite{DSII} assumes the metrized 
log terminal singularities 
appear in the non-collapsed 
Gromov-Hausdorff limits (polarized 
limit space) of K\"ahler-Einstein manifolds. 
After \cite{DSII}, more algebraic works by 
\cite{Li, LL, CS, CS2} appear, which refines 
and more algebraize the conjectural picture which is now a theorem by 
\cite{LX20, LX, LWX, BLnew, XZ, XZ2} (cf., 
the surveys \cite{LLX, Zhuangsurvey}). 

We briefly review this whole story in this subsection. 
In the work of \cite{DS, DSII}, the following setup is 
mainly considered (although being somewhat more general). 

\begin{Setup}\label{GH.setup}
Consider a flat projective family of 
$n$-dimensional smooth polarized varieties 
$f\colon (\mathcal{X},\mathcal{L})\to B$ over a connected 
base algebraic $\k$-scheme $B$ with a section 
$\sigma \colon B\to \mathcal{X}$ which satisfies 
$K_{\X}\equiv_B \mathcal{L}^{\otimes k}$ 
with a constant $k\in \Q$ 
\footnote{\cite{DS, DSII} only requires 
to fix $n$ and the constancy of the volume of the fiber 
$\mathcal{L}_b^n$ for each $b\in B$.} 
Take a sequence of 
points $b_i \in B (i=1,2,\cdots)$ 
whose $f$-fibers are 
$(p_i=\sigma(b_i)\in X_i,L_i)$ and 
admit K\"ahler-Einstein metrics with the K\"ahler class 
in $2\pi c_1(L_i)$. 

We take a sequence of 
compact K\"ahler-Einstein manifolds 
$(p_i=\sigma(b_i) \in X_i,L_i,g_i,\omega_i)
_{i=1,2,\cdots}$ which satisfies 
the {\it non-collapsing condition}: 
\begin{quote}
there is a positive constant $c$ such that 
if we consider the geodesic ball 
$B_r(p_i)$ in $X_i$ 
with center $p_i=\sigma(b_i)$ of radius $r\in [0,1]$ 
we have 
\footnote{(or equivalently the same for $r$ 
in $[0,\epsilon]$ for small 
fixed $\epsilon$ independent of $i$)}
\begin{align}\label{cond:nc}
\vol(B_r(p_i))\ge c r^{2n} \text{ for }i\gg 0. 
\end{align}
\end{quote}

Following may be worth noted. 
\begin{lem}[non-collapsing condition as 
log-terminality]
If $f$ extends to a locally stable projective family 
$\overline{f}\colon 
\overline{\mathcal{X}}\to \overline{B}$ whose base 
$\overline{B}$ contains $B$ as a Zariski open subset and 
restricts to 
$f$ over $B$ i.e., $f^{-1}(B)=\X$, 
and $b_i$ converges to a point $b\in \overline{B}$. 
We further assume that $f^{-1}(b)$ is K-polystable. 
Then, the above non-collapsing condition \eqref{cond:nc} 
holds if and only if $\sigma(b)\in f^{-1}(b)$ is 
log terminal. 
\end{lem}

\begin{proof}
The only if direction follows from \cite{DS} and 
the characterization of log terminality as local volume boundedness 
of the adapted measure in \cite{EGZ}. 
If $k<0$, where $k$ is as in the Setup 1 above, 
then both conditions automatically hold. In the case $k=0$, 
the if direction also follows from \cite{EGZ, DGG} 
(see also \cite{Od21}). 
In the case $k>0$, it follows from \cite{BG, JianSong, DGG}. 
\end{proof}
    
\end{Setup}

\begin{Thm}[\cite{DSII}]

\begin{enumerate}
    \item ({\it loc.cit} 1.1) 
    Passing to a subsequence of $i$, there is a polarized 
    limit space\footnote{the enhanced notion of pointed Gromov-Hausdorff limit to consider complex structure as defined in 
    \cite{DS, DSII}}
    $Z\ni p$ with a singular K\"ahler metric $g$ as a complex analytic space. 
    \item  ({\it loc.cit} 1.3) \label{rev:DSII3}
    Suppose we take a sequence of $(p\in Z,ag)$ with $a\to \infty$, then there is a unique polarized limit space as 
    a log terminal affine variety $C$ with Ricci-flat K\"ahler cone (singular) metric on which a positive dimensional 
    algebraic $\C$-torus $T(\C)$-acts. 
    \item  ({\it loc.cit} \S 2.3, \S3) 
    The above construction of $C$ from $Z\ni p$ can be 
    separated into $2$-steps i.e., 
    $(Z\ni p)\rightsquigarrow W\rightsquigarrow C$, 
    where $W$ is a Fano cone. 
\end{enumerate}
\end{Thm}

The above $C$ is called {\it (local) metric tangent cone} of 
$p\in Z$ with respect to $g$. There are also related results in \cite{VdC}. 
The above item \eqref{rev:DSII3} above 
inspired the later developments in particular by the 
following conjecture. 

\begin{conj}[{\cite[3.22]{DSII}}]\label{DSII.conj}
For any log terminal analytic space $x\in X$ 
with a weak K\"ahler-Einstein metric, 
the metric tangent cone 
only depends on the complex analytic germ of 
$x\in X$ (or $\widehat{\O_{X,x}}$, equivalently). 
\end{conj}

In the next subsection, we also review the developments 
after the above conjecture. 


\subsubsection{Normalized volume \cite{Li}}\label{sec:normvol}
Motivated by the above Donaldson-Sun theory, 
there were developments of 
algebro-geometric machinery which in particular 
partially solved the above conjecture under Setup 1. 
Here we review the developments. 
Henceforth, we fix a closed point $x\in X$ as a $n$-dimensional 
klt singularities over $k$, on which we put the 
trivial valuation. 
We consider the space 
${\rm Val}_k(\widehat{\O_{x,X}})$ of valuations $v$ of 
the completed stalk $\widehat{\mathcal{O}_{X,x}}$ 
whose center is the maximal ideal $\mathfrak{m}_{X,x}$. 
Note that it in particular induces a 
valuation of $X$ centered at $x$. The space of 
such valuations are denoted by ${\rm Val}_x(X)$.

\begin{defn}[Local volume]\label{loc.vol}
(\cite{ELS03}, \cite{LM09})
For each $v\in {\rm Val}_k(\widehat{\O_{x,X}})$, 
we define the following. 
\begin{align*}
{\rm vol}_x(v):=\limsup_{m\to \infty} \dfrac{\dim_k(\O_{X,x}/\{f\mid v(f)\ge m\})}{m^n/n!}
\end{align*}
This $\limsup$ is known to be $\lim$ (cf., 
{\it loc.cit}). 
\end{defn}

In the case $X$ has a structure of affine cone and 
$v$ comes from a positive vector field, 
recall that Lemma \ref{lem:conv} \eqref{lem:conv:eq} 
(from \cite{MSY1,CS}) gives 
another expression of this 
(unnormalized) volume function $\vol(-)$ 
in terms of multi-Hilbert function. 

On the other hand, we take the subspace of 
the space 
${\rm Val}_x(X)$ 
(resp., ${\rm Val}_k(\widehat{\O_{x,X}})$) 
defined by $A_{X}(v)=n$ 
(cf., \cite{MSY2, CS2}) 
and write as 
${\rm Val}_x^n(X)$ 
(resp., ${\rm Val}_k^n(\widehat{\O_{x,X}})$). 

\begin{defn}[\cite{Li}]\label{def:nv}
Define the {\it normalized volume} of $x\in X$ for 
$v\in {\rm Val}_k(\widehat{\O_{x,X}})$ as 
\begin{equation}
\widehat{{\rm vol}}_x(v):=A_X(v)^n {\rm vol}_x(v). 
\end{equation}
Here, if $A_X(v)=+\infty$, 
then we also set $\widehat{{\rm vol}}_x(v):=+\infty$. 
Equivalently, we consider ${\rm vol}_x(v)$ 
on the normalized section 
${\rm Val}_k^n(\widehat{\O_{x,X}})$. 
\end{defn}

By the effect of normalization 
using the log discrepancy $A_X$, for any positive real number $\lambda$, 
\begin{align}
&\widehat{\vol}(\lambda v)=\widehat{\vol}(v)
\end{align}
so that the normalized volume gives a function 
\begin{align}
\widehat{{\rm vol}}\colon ({\rm Val}_k(\widehat{\O_{x,X}})\setminus \{v_{\rm triv}\})/\R_{>0}\to 
\R\cup \{+\infty\}. 
\end{align}
Here, $v_{\rm triv}$ is the trivial valuation and the action of 
$\R_{>0}$ on $({\rm Val}_k(\widehat{\O_{x,X}})\setminus \{v_{\rm triv}\})$ is simply 
given by the multiplication on the real valuations. 
The quotient 
\begin{align}
 ({\rm Val}_k(\widehat{\O_{x,X}})\setminus \{v_{\rm triv}\})/\R_{>0}
\end{align}
or homeomorphic ${\rm Val}_k^n(\widehat{\O_{x,X}})$
is often called {\it non-archimedean link} 
after the topological theory of links of complex algebraic singularities 
(\cite{Sasaki, Brieskorn, Milnor, BoyerGalicki}). 
We note that clearly the latter normalization convention 
is motivated by the Theorem \ref{SE.property} \eqref{gauge.fix}.

\begin{Rem}
We also note that if $v$ comes from the 
plt blow up $\pi\colon X'\to X$ with the exceptional divisor 
(Koll\'ar component\footnote{in the sense of \cite{LX} henceforth, see also 
the original \cite[3.1]{Shokurov}, \cite[2.1]{Prokhorov}}) $E$, 
$\vvol(v)$ equals the generalized log Donaldson-Futaki invariant 
${\rm DF}(X'\to X,-(K_{X'/X}+E))$ defined in 
\cite{Od15}. In the theory, blow up $\pi$ is 
regarded as an analogue of test configuration in 
the sense of \cite{Don02}. 
\end{Rem}

\subsection{Donaldson-Sun 
$2$-step degeneration and its algebraization}

\subsubsection{General procedure}

By making use of the theory of 
normalized volumes in the previous section, 
the Donaldson-Sun $2$-step degeneration theory 
is now put in a format of algebraic geometry and developped, 
due to the work of 
\cite{Li, LL, Blum, CS, CS2, LX, LWX, Xu, XZ, XZ2, BLnew}. 
We briefly summarized the conclusion as follows. 

\begin{Thm}\label{rev:algDS}
Suppose $\k$ is algebraically closed field of characteristic $0$. 
We fix a pointed $n$-dimensional log-terminal 
$\k$-variety 
$x\in X$, consider the stalk 
$\O_{X,x}$ and completed stalk $\widehat{\O_{X,x}}$ 
(Only the formal germ of $x$ matters below, hence $X$ is not 
assumed to be proper in general). 
Then we have the following. 
\begin{enumerate}

\item (\cite{Blum, Xu, XZ2, BLnew})
There is a unique (necessarily quasi-monomial \cite{Xu}) 
valuation $v=v_X\in {\rm Val}_k(\widehat{\O_{x,X}})\subset 
{\rm Val}_x(X)$ up to the $\R_{>0}$-action, 
which minimizes 
\footnote{analogue and the generalization of the 
volume minimization in \cite{MSY1, MSY2} 
(called ``Z-minimization" in {\it loc.cit})} 
the normalized volume $\vvol(-)$. 
We denote the rank of (the groupification $M$ of) 
$\Gamma:={\rm Im}(v_X)$ as $r$ and 
the dual lattice as $N:={\rm Hom}(M,\Z)$. 

\item (\cite{LX20, LX, XZ})\label{rev:algDSW}
${\rm gr}_v(\mathcal{O}_{X,x})
:=\oplus_{s\in \Gamma}
(\{f\in \mathcal{O}_{X,x}\mid v(f)\ge s\}/
\{f\in \mathcal{O}_{X,x}\mid v(f)> s\})$ 
is of finite type over $\k$ (\cite{XZ} for general $r\ge 1$ case). 
There is a natural action of 
the algebraic $\k$-torus $T:=N\otimes \G_m$ on 
$W:=\Spec({\rm gr}_v(\mathcal{O}_{X,x})$ and this 
provides K-semistable Fano cone structure to 
$W$ (\cite{LX20, LX}). 
We denote its vertex as $x_W$.

\item (\cite{LWX}) There is a (non-canonical) affine 
test configuration of 
$W$ to a unique K-polystable Fano cone $C$. 
In particular, $C$ is uniquely determined by 
germ of $x\in X$ (i.e., 
Conjecture \ref{DSII.conj}
by 
Donaldson-Sun \cite{DSII} holds in the case of Setup \ref{GH.setup}). 

\item (\cite{DS,CS,LWX,Li21})
If $\k=\C$, $C$ has a (weak) 
Ricci-flat K\"ahler cone metric. Further,  
if $X$ is compactified to the (non-collapsed) 
polarized limit space of 
K\"ahler-Einstein manifolds in the sense of \cite{DS}, 
\cite[\S 2.1]{DSII}, 
then $C$ is a metric tangent cone of $X\ni x$ 
and is unique. 

\end{enumerate}
\end{Thm}



\subsubsection{Local vs global volume}
The following type of inequalities are very important to 
study K-stability of singular $\Q$-Fano varieties. 
In the dimension $2$, \cite{OSS} originally used the 
Bishop-Gromov type inequality to deduce a similar (weaker) result which we also review. 

\begin{Thm}[\cite{Liu}]\label{Liu.vol}
For any K-semistable $\Q$-Fano variety $X$ 
and $x$ its closed point, we have 
\begin{align}
(-K_{X})^n\le 
\biggl(1+\frac{1}{n}\biggr)^n\widehat{\rm vol}(v). 
\end{align}
\end{Thm}

Here is a weaker version, which in turn is a straightforward 
consequence of the Bishop-Gromov inequality for orbifolds. 

\begin{Thm}[cf., \cite{OSS}]
If there is a quotient singularity of the type 
$\C^n/\Gamma$ ($\Gamma$ is a finite subgroup of ${\rm GL}(n,\C)$) 
on a 
K\"ahler-Einstein $\Q$-Fano variety $X$, we have 
\begin{align}
(-K_{X})^n\le \frac{1}{|\Gamma|}\cdot \dfrac{2 (2n-1)^n\cdot n!}
{(2n-1)!!=(2n-1)\cdot (2n-3)\cdots 3\cdot 1}. 
\end{align}
\end{Thm}
The estimation comes from the comparison of the (real) space 
form, 
rather than the complex space form $\P^n$, 
on which we can not expect (complex) 
algebraic variety structure in general. 
\begin{proof}
The proof in the $2$-dimensional case (\cite[2.7]{OSS} cf., also 
\cite{Tiansurf}\footnote{cf., Remark 2.8 of \cite{OSS}}) 
naturally 
extends. 
Apply (orbifold version of) the 
Bishop-Gromov comparision theorem 
to $X$ to obtain 
\begin{align*}
\frac{(2\pi)^n(-K_X)^n /n!}{\vol(S^{2n}(2n-1))}\le \frac{1}{|\Gamma|}. 
\end{align*}
Hence the assertion follows from 
\begin{align*}
\vol(S^{2n}(2n-1))&=(2n-1)^n\cdot \frac{(2n+1)\pi^{n+\frac{1}{2}}}{\Gamma(n+\frac{3}{2})=(n+\frac{1}{2})(n-\frac{1}{2})\cdots \frac{1}{2}\times \sqrt{\pi}}\\
			  &=(2n-1)^n\cdot \frac{2 \pi^n}{(n-\frac{1}{2})(n-\frac{3}{2})\cdots \frac{1}{2}}.\\
\end{align*}
\end{proof}

\subsubsection{Toric case}

We can see the strength of the K-stability theory and  
Donaldson-Sun theory by one of simplest example - 
toric case. 
As first pointed out by \cite[\S 1]{CS2}, 
it readily re-proves 
a somewhat weaker version of 
the following celebrated result 
by Futaki-Ono-Wang \cite{FOW} and its generalization to 
singular case \cite{Berman.Sasaki}. 
Note that the condition in {\it loc.cit} 
Theorem 1.2, the meaning of the assumptions 
in {\it loc.cit} 
Theorem 1.2 
(on the vanishing of 
the de Rham cohomology class of the contact bundle 
and the positivity basic first Chern class of the normal bundle of the Reeb foliation) is later 
clarified to be equivalent to 
the ($\Q$-)K-triviality which 
matches Definition \ref{def:cone} and 
Lemma \ref{lem:cK}. Hence, we restate their result in 
that manner, although in a somewhat weak statements 
(on the stability side). 

\begin{Thm}[cf., \cite{FOW}, \cite{Berman.Sasaki}]
Suppose $T\curvearrowright X=C(S)$ is a toric Fano cone i.e., 
a (log terminal) 
Fano cone which admits an action of algebraic $\k$-torus $T'$ 
which preserves the structure. We denote 
$N':={\rm Hom}(\G_m,T')$ and $M':={\rm Hom}(T',\G_m)$. 
Then $X$ is K-semistable and 
$T'$-equivariantly K-polystable as Fano cone with respect to a 
(unique) Reeb vector field $\xi\in N'\otimes \R$. 
\end{Thm}

\begin{proof}
The (unique) $\vvol$-minimizing valuation 
$v=v_X$ of $\O_{X,x}$ 
is $T'$-invariant hence is a toric valuation 
in the sense of \cite[\S 8]{Blum}. Thus, 
if we write $C(S)=\Spec R=\oplus_{\vec{m}\in M'}R_{\vec{m}}$ 
for $M':={\rm Hom}(T',\G_m)$ 
as the decomposition of the coordinate ring 
by the $T$-action, it is easy to see ${\rm gr}_v(R)=R$ 
so that $W=X$ in the Theorem \ref{rev:algDS}. 
In particular, $X$ is K-semistable with respect to a 
Reeb vector field in $N\otimes \R$ and $T\subset T'$. 
Further, if there is a $T'$-equivariant special test 
configuration $(\X,\xi,\eta)$, then since 
$\dim(X)=\dim(T')$, $T\times \G_m$ can not act 
effectively on any component of 
$\X_0$ so that $\X$ is a $T'$-equivariant 
product test configuration. Hence, because $\X_0$ has 
vanishing Futaki invariant with respect to 
$\xi$ by its K-semistability, it follows that 
$X$ is $T'$-equivariantly K-polystable with respect to $\xi'$.    
\end{proof}

This contrasts with the fact that toric Fano manifolds 
do not necessarily have K\"ahler-Einstein metrics, but 
this is due to the flexibility of the (abstract) Reeb 
vector fields.

\subsection{Affine generalized test configurations vs  
filtered blow ups}

As a 
preparatory general material, we excerpt a part of 
\cite{Od24b} on the theory of generalized 
test configurations and 
give 
an alternative description 
of affine generalized 
test configuration for general affine varieties 
(cf., also \cite{Inoue4, BJI}). In short, generalized test configurations 
correspond to certain ideals 
sequences or their (filtered) blow ups. 
This generalizes the perspective of \cite{OdDF}. 

Firstly, we recall the definition of 
generalized test configurations.

\begin{defn}
For an affine toric variety $U_\tau$ 
which corresponds to a rational 
polyhedral cone $\tau\subset N\otimes \R=:N_{\R}$ as 
in Notation \ref{ntt}, 
a {\it generalized test configuration} of affine variety 
$X$ is a 
faithfully flat 
$T$-equivariant affine morphism $p\colon \mathcal{Y}\twoheadrightarrow 
U_{\tau}$ with general fiber $X$. 

If general fibers are 
affine cones in the sense of 
Definition \ref{def:cs} with respect to 
a $p$-fiberwise action 
of an additional algebraic $\k$-torus $T'$, 
we further require the following: 
\begin{quote}
$\mathcal{Y}=\Spec \mathcal{R}$ 
with a finite type $\Gamma(U_\tau)$-algebra $\mathcal{R}$ 
that decomposes as 
$\mathcal{R}=\oplus_{\vec{m}\in M}\mathcal{R}_{\vec{m}}$,  
to the $T'$-eigen-$\Gamma(U_\tau)$-submodules, 
then $\mathcal{R}_{\vec{m}}$ for each $\vec{m}$ 
is a locally free module. 
\end{quote}
\end{defn}

\begin{ex}
For a klt affine variety $X\ni x$, 
take the valuation $v_X$ of minimizing normalized volume  
(\cite{Li, Blum}) 
and set $W$ as $\Spec ({\rm gr}_{v_X}(\mathcal{O}_{X,x}))$. 
Denote the groupificiation $N$ of ${\rm Im}(v_X)$ and a  
rational polyhedral cone $\tau\subset N_{\R}$ 
which includes $\xi$. 

Naturally, 
there is a canonical generalized test configuration 
$\pi_{\tau}\colon \mathcal{X}_{\tau}\twoheadrightarrow 
U_{\tau}$ of $X$ whose fiber over $p_\tau$ is $W$ 
(\cite{Teissier0, LX}). 
This gives a positive weight deformation of $W$ 
in the sense of 
\cite{Od24b}. 
See more details in \cite[\S 2]{Od24b} (cf., also 
\cite{LX}). 
\end{ex}

\begin{lem}\label{lem:idseq}
Fix an affine $\k$-variety $\Spec(R)=X$ and a rational polyhedral cone $\tau\subset N\otimes \R$ as before. 
Then, the following two sets have natural bijective correspondence with each other: 

\begin{enumerate}
\item \label{gen..tc}
the set of isomorphic classes of 
(affine) generalized test configuration over $U_\tau$ 
which dominates
\footnote{analogous to the blow up type (semi) 
test configurations  
in \cite{Od13}}
$X\times U_\tau$, 
denoted by 
$\pi\colon \mathcal{X}=\Spec(\mathcal{R})\twoheadrightarrow \Spec(\k[\mathcal{S}_\tau])$ (here, $\mathcal{S}_\tau:=\{x\in M\otimes \R\mid \langle x,\tau\rangle 
\subset \R_{\ge 0}\} \cap M$) 

\item \label{idseq}
the set of 
sequence of 
ideals 
 of $R$ which we write 
$\{I_{\vec{m}}\subset R\}_{\vec{m}\in M}$ 
(or $\vec{m}\mapsto I_{\vec{m}}$) 
that 
satisfies the following: 
\begin{itemize}
\item 
$I_{\vec{m}}\cdot I_{\vec{m}'}\subset I_{\vec{m}+\vec{m}'}$ for any $\vec{m}, \vec{m}' \in M$, 
\item $I_{\vec{m}}\subset I_{\vec{m}'}$ 
for any $\vec{m}, \vec{m}' \in M$ with 
$\vec{m}'-\vec{m}\in \mathcal{S}_\tau$, 
\item and $\oplus_{\vec{m}}I_{\vec{m}}\cdot \chi(\vec{m})(=\mathcal{R})$ 
is of finite type over $\k$ (or equivalently, over $\k[\mathcal{S}_{\tau}]$), 
\item 
$I_{\vec{m}}\neq R$ for (some or any) 
$\vec{m}\in -\mathcal{S}_{\tau}^o$. 
\end{itemize}
\end{enumerate}

\eqref{idseq} is analogous to the description of 
equivariant toric vector bundles by family of filtrations 
in \cite{Kl90} 
(see also \cite[\S 2.2]{Inoue}, \cite[A3]{BJI}). 
Indeed, for each $\xi\in \tau$ 
and each $k\in \R_{\ge 0}$, one can define 
graded sequence of ideals as 
\begin{equation}\label{familyfiltration}
\mathfrak{a}_{\xi, k}:=
\sum_{\vec{m},\langle \vec{m},\xi\rangle \le -k}
I_{\vec{m}}. 
\end{equation}
Furthermore, below we assume $X$ is normal 
and consider the following sets (3). 
Then, 
if $\xi\in\tau^o\cap N_{\Q}$, 
the above set \eqref{gen..tc} with normal $\X$ 
(or equivalently, \eqref{idseq} with normal $\mathcal{R}$) 
has a natural map to the set (3) for each such $\xi$. 

\begin{enumerate}
\item [(3)] 
the relative isomorphism class (over $X$) of 
a birational projective morphism 
$h\colon Y_\xi \to \Spec(R)$ with an 
effective $h$-antinef 
$\Q$-Cartier $\Q$-divisor $E_\xi$. 
\end{enumerate}
Further, if we consider the base change of $\pi$ 
via the toric morphism $\A^1\to U_\tau$ which corresponds to $l\xi$ and denote 
as $\X_{l \xi}\to \A^1$, 
$(\X_{l \xi},(\X_{l \xi})_0)$ 
is log canonical if and only if 
$(Y_\xi,E_{\xi})$ is log canonical. 
\end{lem}

\begin{proof}
We first prove the bijection between the first two i.e., \eqref{gen..tc} and \eqref{idseq}. 
Since $\mathcal{R}$ is flat over $\k[\mathcal{S}_\tau]$, 
$\mathcal{R}\subset \mathcal{R}\otimes_{\k[\mathcal{S}_\tau]} \k[M]\simeq R[M]$. 
We put 
$I_{\vec{m}}:=
\{f\in R\mid \chi(\vec{m})f\in \mathcal{R}\}$ for 
each $\vec{m}\in M$. 
The condition that $\X$ dominates 
$X\times U_\tau$ ensures that $I_{\vec{m}}$ 
are ideals of $R$ and the 
last condition of $I_{\vec{m}}$ in the item 
\eqref{idseq} 
is equivalent to the surjectivity of $\pi$ i.e., 
$\mathcal{R}\neq (\oplus_{\vec{m}\in \mathcal{S}_\tau \setminus 0}R)\mathcal{R}$. 

Next we discuss the map from 
\eqref{idseq} with normal $\mathcal{R}$ 
to the set (3). Fix $\xi$ and consider its multiple $l\xi \in N$. 
The proof of \cite[2.20, 2.21]{LWX} applies to obtain a blow up 
$h \colon 
Y_\xi \to X$. 
Further this construction naturally associates 
an effective $h$-antiample 
$\Q$-divisor $E_{\xi'_q}$. By {\it loc.cit}, the last assertion of the 
equivalence of log canonicities also follows. 
We complete the proof. 
\end{proof}


\subsection{(Higher) $\Theta$-stratification}

One of the main tools in the algebraic construction of K-moduli of Fano varieties 
(cf., \cite{Xu, BHLLX, LXZ}) is stack-theoretic incarnation and generalization of the 
theory of Harder-Narasimhan filtration, which is introduced by \cite{HL} and 
called the {\it $\Theta$-stratification}. Roughly put, it allows a family-wise 
optimal destabilization along one parameter family. 
Correspondingly, the properness-type criterion after Langton \cite{Langton} 
is established in \cite{AHLH}. 
Because of the irrational nature of 
$\vvol$-minimizing valuations (associated to 
the Reeb vector fields), 
it is more convenient and suitable to 
use a generalization of \cite{HL, AHLH} 
to multi-variable parameter family setup and degenerations in its irrational direction, 
which is done as a part of a theory of 
\cite[\S 2, Ex 2.10]{Od24b}. 
We briefly review some of the 
main definitions and theorems below, in a simplified 
weaker form, 
following the toric Notations \ref{ntt}. 

\begin{defn}[\cite{HL, Od24b}]\label{higher.strata}
Generalizing \cite{HL} ($r=1$ case), a 
 {\it higher $\Theta$-stratum} of rank $r$ and type 
 $\tau$ in 
$\mathcal{M}$ 
consists of a union of connected components $\mathcal{Z}^{+}\subset 
{\rm Map}(\Theta_\tau, \mathcal{M})$, where ${\rm Map}(-)$ denotes the mapping stack 
(cf., \cite{AHR, AHR2}), such that the natural evaluation morphism 
${\rm ev}_{(1,\cdots,1)} \colon \mathcal{Z}^{+} \to \mathcal{M}$ 
is a closed immersion. 
\end{defn}

\begin{defn}[\cite{HL, Od24b}]
For a quotient algebraic $\k$-stack by a linear algebraic $\k$-group, a (finite) 
 {\it higher $\Theta$-stratification} in 
$\mathcal{M}$ 
consists of 
\begin{enumerate}
\item a finite set of real numbers $1\in \Gamma \subset (-\infty,1]\subset \R$, 
\item closed substacks 
 $\mathcal{M}_{< c}$ of $\mathcal{M}$ for $c\in \Gamma$ which are monotonely enlarging i.e., 
$\mathcal{M}_{< c}\supset \mathcal{M}_{< c'}$ if $c> c'$ in $\Gamma$. 
Naturally, we can define substacks of $\mathcal{M}$ 
as $\mathcal{M}_{>c}, \mathcal{M}_{=c}, 
\mathcal{M}_{\ge c}, \mathcal{M}_{\le c}$. 
\item 
higher $\Theta$-strata structure of 
type $N, \tau$ 
(Def \ref{higher.strata}) 
on each $\mathcal{M}_{=c}$ for each $c\in \Gamma$. 
We do not require $N,\tau$ to be identified for 
different $c$. 
\end{enumerate}
\end{defn}

For Definition \ref{higher.strata}, we have the following Langton type theorem, for instance. 
Clearly, we can use it iteratively to 
obtain a generalization of \cite[6.12]{AHLH} for higher 
$\Theta$-stratification. 

\begin{Thm}[Higher $\Theta$-stable reduction \cite{AHLH, Od24b}]\label{rev:gAHLH}
For any quotient algebraic stack  $\mathcal{M}=[H/G]$ 
with finite type scheme $H$ and linear algebraic group $G$ 
over $\k$, 
consider any morphism $f\colon \Delta\to \mathcal{M}$ 
from $\Delta=\Spec \k[[t]]$, 
to $\mathcal{M}$ 
whose closed point $c$ 
maps into a closed substack $\mathcal{Z}^+$ corresponding higher $\Theta$-strata for the cone $\tau\ni \xi$, 
while the generic point maps  
outside $\mathcal{Z}^+$. 
We denote the restriction of $f$ to $\Spec(K)$ as $f^o$. 

Then, 
after a finite extension of $R$ and shrinking $\tau$, 
there is a toric morphism $e\colon U_\tau\to \A^1$ 
(we denote its complete localization 
$\Spec \widehat{\mathcal{O}}_{U_\tau,p_\tau}\to 
\Spec (\widehat{\mathcal{O}}_{\A^1,0}=\k[[t]])$ 
also as $e$) 
and a modification of $f$ as 
$$f|_\xi \colon \Spec \widehat{\mathcal{O}}_{U_\tau,p_\tau} 
\to \mathcal{M},$$
which extends $f^o\circ e$ and 
sends $p_\tau(\kappa)$ to a point outside $\mathcal{Z}^+$, 
but still isotrivially degenerates to a point in 
$\mathcal{Z}^+$. 
\end{Thm}
The original statements in \cite{Od24b} 
are somewhat stronger, more general and canonical, 
for some irrational element $\xi\in \tau$. 
We use the above theorem \ref{rev:gAHLH} 
both for the following construction of K-moduli of 
Calabi-Yau cones in the next section \S 
\ref{sec:moduli} and some other later work. 
See also \cite[Appendix]{Od24c} for another more scheme theoretic 
explanation 
on this higher $\Theta$-stable reduction and applications. 


\section{Moduli of Calabi-Yau cones}\label{sec:moduli}

Now we consider the moduli of K-polystable Fano cones, 
i.e., Calabi-Yau cones, 
depending on the various preparations in the previuos section. 

\subsection{Boundedness}\label{sec:bdd}

In $n=2$ (\cite{HLQ, LMS}) and $n=3$ (\cite{LMS, ZhuangII}) case, the following boundedness result is 
proved, and is now generalized to any $n$ by 
\cite{XZ24}
\footnote{we learnt this result 
after the completion of our manuscript} 
after the corresponding results and arguments  
in log Fano varieties case 
\cite{Jiang}. 

\begin{Thm}[boundedness \cite{XZ24}]\label{bd}
For a fixed positive integer $n$, 
$n$-dimensional K-semistable 
$\mathbb{Q}$-Fano cone $X$ of the 
normalized volume at least $V$ 
are bounded. 
\end{Thm}

\begin{cor}\label{bdcor1}
  In particular, for each $n$ and $V$, 
there are only finitely many choices of 
the (metric) Reeb vector fields with minimum normalized volume.  
  \end{cor}

Recall that multi-Hilbert schemes \cite{multiHilb} 
parametrizes affine closed schemes inside $\A^N$ with a linear action of 
an algebraic torus $T$ with fixed multi-Hilbert series. 
Using this, one can rephrase the above theorem: 

\begin{cor}\label{bdcor2}
For fixed $n,V$, $n$-dimensional K-semistable 
$\mathbb{Q}$-Fano cone $X$ of the 
normalized volume at least $V$ (for fixed $n, V$) 
are parametrized inside a finite union of certain 
projective 
multi-Hilbert schemes for all positive weights. 
\end{cor}

\begin{proof}[proof of Corollaries]
We write the  
proof for the convenience. 
For each $n$-dimensional K-semistable 
$\mathbb{Q}$-Fano cone 
$X=\Spec \oplus_{\vec{m}}R_{\vec{m}}$ 
of the normalized volume at least $V$, 
take a positive vector field $\xi$ i.e., with 
$\langle \vec{m},\xi\rangle >0$ 
unless $R_{\vec{m}}=0$. By Theorem \ref{bd}, 
combined with the latter half arguments of the proof of \cite[Lemma 3.14]{Od24b}, 
we can take a uniform finitely generated regular submonoid $\Gamma_{\ge 0}\subset M$ 
generated by the set of the extremal integral vectors $S\subset \Gamma_{\ge 0}$ 
which contains all the moment monoids of $X$ and their $\{R_{\vec{m}}\}_{\vec{m}\in S}$ 
generate the coordinates ring $R$ of $X$. $S$ gives uniform embedding of 
$X\to \A^N$ and for small enough rational polyhedral cone $\tau\subset N\otimes \R$ 
which includes $\xi$, we have strict positivity $\langle \tau, S\rangle \subset \R_{>0}$. 
This gives an isomorphism $T\simeq \G_m^r$ and 
so that the weights on $\A^r$ of each $\G_m$ are all positive. 
In particular, by \cite[Theorem 1.1, Corollary 1.2]{HS}, there is a {\it projective} 
multi-Hilbert scheme (possibly non-connected) which parametrizes all $n$-dimensional 
K-semistable Fano cone $X$ of the normalized volume at least 
$V$ and the Reeb vector field $\xi$. 
The finiteness claim of the Reeb vector 
fields follows from the 
characterization of the K-semistability 
\cite{CS} 
in terms of the 
minimized normalized volumes 
\cite[1.1]{LX}, combined with 
the constructible semi-continuity 
of the minimized normalized volumes 
\cite{BlumLiu, Xu} (See also related arguments 
later during the proof of Theorem \ref{A1}). 
\end{proof}

We consider the obtained 
finite union of projective 
multi-Hilbert schemes and take its 
union of components which parametrize those with 
normalized volume exactly $V$, and denote by  
$H_{V}$ in our paper (though it also depends on $n$ in general, 
just for simplicity). We sometimes fix $\xi$ among the 
finite choices. 
Note that the leading coefficient $a_0(\xi)$ of  $F(\xi,t)$ is $V$ because of Theorem \ref{volvol}. 
Hence, in particular, for any pointed log terminal variety $x\in X$ parametrized in $H_V$, 
the normalized volume $\widehat{\vol}(x\in X)$ is at most $V$ from the definition. 

Below, 
due to the finiteness of $\xi$, 
we can and do fix it in addition to the fixing of 
volume $V$. 
We then replace $T$ by the minimum 
algebraic subtorus which contains $\xi$ 
in its Lie algebra, if necessary and set 
$G$ to be is the commutator of $T\subset {\rm GL}(N)$, which is reductive. 
We then consider 
locus $H_{V}^{\rm ss}(\xi):=\{b\in H_V\mid 
0\in X_b \text{ is K-semistable with respect to }
T, \xi\}$ and replace $H_V$ by 
its closure. Note that naturally 
the reductive group $G$ preserves 
$H_{V}^{\rm ss}(\xi)\subset H_V$ 
(see e.g., \cite{multiHilb, DSII}). 

From now on, we want to prove that the locus of K-semistable Fano cones of 
the normalized volume $V$ and the Reeb vector 
field $\xi$, which we later denote as 
$B_{V}^{\rm ss}(\xi)$ makes sense and 
the corresponding quotient stack $[B_{V}^{\rm ss}(\xi)/G]$ 
admits a good moduli algebraic space which is proper. 

\subsection{Locally closedness}\label{sec:loccl}

We consider the universal family over $H_V$ which we denote as 
$(\A^N\times H_V) \supset \U_V \twoheadrightarrow H_V$ with the fixed multi-Hilbert series. Since the weights 
of the action of $T=\G_m^r \curvearrowright \A^N$ are all 
positive, the fibers contain the origin i.e., $\U_V\supset 0\times H_V$ 
and the action are good in the sense of \cite{LS13} (if the fibers are normal). 

We take the locus $B'_V$ of $H_V$ which parametrizes normal 
and $\Q$-Gorenstein fibers 
by Koll\'ar's hull and husk \cite{hull, Kol22}. 
(There will be also remarks later on the subtleties on 
definition of corresponding families and their 
logarithmic generalizations: 
Remarks \ref{family.subtle}, \ref{family.subtle.log}.) 
Further, inside $B'_V$, we take 
the open locus $B_V$ where the vertices are 
log terminal, due to its openness. 

By the standard fact (cf., e.g., \cite[Lemma 3.1]{Kol13}), we see that 
for any algebraic 
subtorus $T'\subset T$ 
of rank $1$, 
the $T'$-quotients 
of the geometric fiber of 
$(\U_V \setminus \H_V)\to \H_V \ni b$ are  log $\Q$-Fano 
varieties precisely when $b\in B_V$, 
with respect to the natural boundary branch $\Q$-divisors. 
Hence, this $B_V$ is exactly the locus which  parametrizes the Fano cones. 
We denote the restriction of the $G$-equivariant universal family to $B_V$ 
as $X_V \to B_V$. We denote an arbitrary geometric $\k$-point 
$b\in B_V$ and the fiber as $X_b$ on which $T$ acts which commutes with 
the $G$-action. 

Note that for any $b\in B_V$, $\widehat{\vol}(X_b)$ is at most $V$ as we explained in the 
previous subsection. 
Therefore, if we apply the semicontinuity of the normalized volume 
\cite{BlumLiu} combined with \cite[1.3]{Xu} 
to the universal family over $B_{V}$ and the Noetherian arguments, 
it follows that $\{\widehat{\vol}(X_b)\mid b\in B_V(\k)\}$ is a finite set 
which we denote as $V=V_0>V_1>\cdots>V_m>\cdots>V_{m'}$. 
Correspondinglly, 
we have a filtration of 
$G$-invariant open subschemes 
$$B_V(v= V)\subset B_V(v\ge V_1)\subset \cdots \subset 
\cdots \subset 
B_V(v\ge V_{m'})=B_V,$$ 
where we mean 

\begin{align*}
B_V(v=V)&:=\{b\in B_V\mid \vvol(0\in X_b)=V\}, \\
B_V(v\ge V_i)&:=\{b\in B_V\mid \vvol(0\in X_b)\ge V_i\}.
\end{align*}
Clearly we have $B_V(v\ge V_0)=B_V(v= V)$ though. 
We sometimes simply denote 
$B_V(v\ge V_i)$ as $B_V(i)$. 

\begin{Claim}\label{ss.vol}
$B_V(0)=B_V\cap H_V^{\rm ss}(\xi)$. 
\end{Claim}
\begin{proof}
For a geometric point $b\in B_V$, 
the Reeb vector field $\xi$ gives a 
valuation $v_\xi$ of $X_b$ with 
$\vvol(v_\xi, X_b)=V$. Therefore, 
if $\vvol(0\in X_b)=V$ i.e., $b\in B_V(0)$ 
if and only if 
the $\vvol$-minimizing valuation of $X_b\ni 0$ 
is exactly $v_\xi$ 
but it is characterized by the K-semistability of 
$(0\in X_b \curvearrowleft T, \xi)$ 
in the sense of \cite{CS, CS2}, 
by \cite[1.1, 1.3]{LX} (cf., also \cite[6.1]{CS}). 
\end{proof}

Hence, we can consider the algebraic stack of 
K-semistable Fano cones 
of fixed normalized volume $V$ and the 
Reeb vector field $\xi$, 
which we denote as 
\begin{align}\label{Mss}
\mathcal{M}^{\rm ss}_V(\xi):=[B^{\rm ss}_V(\xi)/G].
\end{align}
We also write $\mathcal{M}_V$ for the bigger quotient stack $[B_V/G]$, 
$\mathcal{M}_V(i):=[B_V(v\ge V_i)/G]$ for each $i=1,2,\cdots,m'$. 

By \cite{Xu} (cf., also \cite{BlumLiu}), 
more precisely 
by its third paragraph of the proof of Theorem 1.3 and its Theorem 2.18, 
the $\vvol$-minimizing valuations is 
uniformly taken i.e., obtained as the restriction of 
the same quasi-monomial valuation of $\O_{0,\A^N}$, 
on each strata of some 
finite stratification $\{B_V((j))\}_j$ 
by some locally closed connected 
subsets $B_V((j))$ of $B_V$. 
Note that {\it loc.cit} crucially depends on the bounded complements by Birkar 
\cite{Birkar-K} as well as an analogue of 
the invariance of local 
plurigenera (cf., \cite{HMX}). 
Theorem \ref{A1}, which we next explain, 
further implies that 
each $B_V((j))$ is a connected component of some 
$B_V(i)\setminus B_V(i-1)$: 

\begin{Claim}\label{Claim:uniform}
On each connected component of $B_V(i)\setminus B_V(i-1)$, 
the $\vvol$-minimizing valuations of $0\in X_b$ for $b\in B_V(i)\setminus B_V(i-1)$ is 
uniformly taken. 
\end{Claim}

We will explain this in details in 
the next section but 
related arguments are also in \cite{Chen, Od24c}. 


\subsection{Higher $\Theta$-reductivity-type theorem}

Now we discuss $\Theta$-reductivity (cf., \cite{HL, AHLH}) 
of 
the moduli stack of Fano cones in a somewhat generalized 
form, to include non-cone affine varieties for later purposes. 
The statement is compatible with the framework of 
\cite{Od24b}, as an irrational and family analogue of the 
theory of the Harder-Narasimhan filtration or 
the $\Theta$-strata \cite{AHLH}. 

\begin{Thm}[{cf., \cite{Chen}}]\label{A1}
Consider an arbitrary 
faithfully-flat 
affine klt morphism $\pi\colon Y \to S$
with a section $\sigma\colon 
S \to Y$
for a reduced algebraic $\k$-scheme $S$ 
over $\k$ of characteristic 
$0$, 
suppose there is a constant $V$ such that 
for any geometric point $s\in \Spec(R)$, 
the geometric fiber $Y_s\ni \sigma(s)$ satisfies $\vvol(\sigma(s)\in 
Y_s)=V$. 

Then there is an algebraic torus $T=N\otimes \G_m$, 
$\xi \in N_{\R}\setminus N_{\Q}$ and a rational 
polyhedral cone $\tau\ni \xi$, 
all independent of $s$, 
such that $\pi$ extends to 
a faithfully flat affine klt 
family $\tilde{\Y}=[\Y/T]$ over the quotient algebraic 
$S$-stack 
$\Theta_\tau \times_{\k} S=[U_\tau(S)/T(S)]$, such that for any $s$, $\tilde{\Y}$ restricts to 
the positive weight deformations  
of $Y_s\ni \sigma(s)$ to 
the K-semistable Fano cones $W_s$ discussed in \S  \ref{sec:DSreview} and \cite[\S 2]{Od24b}. 
\end{Thm}

{\textit{Comments on the history}} 
A half month after we posted our first version on arXiv with the above statements, 
\cite{Chen} appears with a similar claim (Main Theorem 1.1 of {\it op.cit}) as above Theorem 
\ref{A1}. It assertion is 
the existence of ideal sequence $\{\mathfrak{a}_{\lambda}\}_{\lambda\in \R_{\ge 0}}$ 
of $\mathcal{O}_Y$ with ${\rm Spec}_S \oplus_\lambda \mathfrak{a}_\lambda/\mathfrak{a}_{>\lambda}$ gives a family of K-semistable Fano cones, which will be 
our desired $\mathcal{Y}|_{p_\tau\times S}$. Indeed, if we take 
${\rm Spec}_S \oplus_{\lambda \in S\subset \R_{\ge 0}} \mathfrak{a}_\lambda$ 
for $S:=\{\lambda\in \R_{\ge 0}\mid \mathfrak{a}_\lambda\neq \mathfrak{a}_{>\lambda}\}$, 
we easily see by the Nakayama's lemma and \cite{Teissier0} 
(or the arguments using universal Gr\"obner basis cf., \cite{Od24b}) 
that it gives our desired $\mathcal{Y}$ of above Theorem \ref{A1}. 
Then, its author also pointed out an inaccuracy in the last step of our original proof of it (i.e., the proof of {\it log terminality} of the obtained special fibers 
after the proof of the necessary finite generation). 
Hence, we simply give up credit on this Theorem \ref{A1} and refer to \cite{Chen} for 
the detailed proof. I thank him for the communication. 
His proof looks different from ours and depends on the construction of family of 
quasi-divisorially-log-terminal anticanonical models (``Koll\'ar models") 
on $\mathcal{Y}$ that encodes the volume-minimizing valuations (after \cite{XZ}). 
We only keep another more elementary proof for 2-dimensional case for now. 
(Other arguments can be also found in version 1 of this paper on arXiv, 
where until the middle of Step (ii) should be still valid.) 

\begin{proof}[Simpler proof of Theorem \ref{A1} for $n=2$]
We first note that $n=2$ case can be checked by the following 
standard arguments, under the assumption that $S$ is a smooth 
$\k$-curve. 
In this case, we take a closed point $c$ and 
suppose that $\sigma(c)$ has $\Q$-Gorenstein index $N$. 
Now we take the canonical cover of $Y$ 
as $\tilde{Y}\to Y$ of the Galois group $\mu_N(\overline{\k})$. 
Suppose that the generic fiber $Y_\eta\ni \sigma(\eta)$ 
(resp., special fiber $Y_c \ni \sigma(c)$ 
for $\pi$ is the 
(quasi-\'etale) 
quotient by a finite group $G_\eta$ 
(resp., $G_c$) 
of the order $a_\eta$ (resp., $a_c$). 
If the preimage of $\sigma(S)$ in $\tilde{Y}$, denoted as 
$\tilde{S}$, has degree 
bigger than $1$ over $S$, we make base change of 
$\tilde{Y}\to S$ by 
$\tilde{S}\to S$ and denote the obtained family 
(resp., section) as $\tilde{\pi}\colon \tilde{Y}\to \tilde{S}$ 
(resp., $\tilde{\sigma}$). 
Then consider the 
pointed $\tilde{\pi}$-generic fiber which we suppose to be the 
quotient by a finite group $\tilde{G}_\eta$ 
(resp., $\tilde{G}_c$) 
of the order $b_\eta$ (resp., $b_c$). 
Note that $a_c\ge a_\eta$, $b_c\ge b_\eta$, 
$a_c=N b_c$, $a_\eta\le N b_\eta$ from the construction. 
From our assumption, we have $a_c=a_\eta=\frac{4}{V}$ i.e., 
the local volume does not decrease at $c$. Combining them, 
we obtain $b_0=b_\eta$ i.e., we can assume that 
$\tilde{\pi}$ is a flat family of ADE singularities whose 
local fundamental groups orders do not change. 
Then it is a classical fact $\tilde{\pi}$ 
is formally trivial 
so that 
the (divisorial) $\vvol$-minimizing valuations 
$v_s$ 
of 
$\sigma(s)\in Y_s=\pi^{-1}(s)$ for closed points of $s\in S$ 
do not jump in the sense that it 
can be realized as a $S$-flat coherent ideal 
on $Y$ supported on $\sigma(S)$ 
(whose blow up gives a plt blow up for each 
$s\in S$, which corresponds to $v_s$). 
Hence, 
Theorem \ref{A1} for $n=2$ follows. 
\end{proof}

In the current paper, we only need the following statements, which 
restricts out attention to the case of 
Fano {\it cones}. 

\begin{prop}\label{corA}
The moduli stack $\mathcal{M}^{\rm ss}_V$ of $n$-dimensional 
K-semistable Fano cones is $\Theta$-reductive in the sense of \cite{HL}. 
\end{prop}

\begin{proof}[proof of Proposition \ref{corA}]
The construction of the moduli stack $\mathcal{M}^{\rm ss}_V$ is done in 
\S \ref{sec:moduli} (more precisely, \S \ref{sec:bdd}, \S \ref{sec:loccl}). 
Take any family of K-semistable Fano cones 
$\mathcal{Y}^o$ 
which comes from 
$(\Theta_{R}\setminus p)\to \mathcal{M}^{\rm ss}_V$, 
where $\Theta_R:=[\Spec R[t]/\G_m]$ and 
$p$ means the closed point for the maximal ideal $(\mathfrak{m},t)$. 
We first restrict $\mathcal{Y}^o$ to the generic fiber and 
then denote the corresponding $\mathbb{G}_m$-action on the central fiber 
$\mathcal{Y}|_{(t)}$ as $\eta$. Then, as in \cite{LWX} etc, we twist the 
$\mathbb{G}_m$-action on the test configuration (without changing the total space), 
so that $\eta$ is replaced by $\eta+m\xi'$ with some positive vector field $\xi'$ 
and divisible enough $m$, 
which is still positive vector field on the fiber of $\mathcal{Y}^o$ over $(t)$. 
Then we obtain a sequence of divisorial valuations $\{E_q\}_{q=1,2,\cdots}$ 
on the generic fiber of $\mathcal{Y}^o$ with the center inside it  as in \cite[\S 2.5.1]{BX}, \cite[p1026]{ABHLX}, 
which can be extracted by a birational morphisms 
$E_q\subset Y_q\to Y$ for $q\gg 0$, 
where $Y$ is the fiber of $\mathcal{Y}^o$ over 
$(t\neq 0)\simeq {\rm Spec}(R)$ by using \cite{BCHM, Blum}. 
Now, we apply the same arguments as \cite[p1028-1029]{ABHLX} 
to obtain an affine faithfully-flat family $\mathcal{Y}\to \Theta_R$ 
which extends $\mathcal{Y}^o$. As in {\it loc.cit}, 
we know that $(Y_q,E_q+Y_{\kappa})$ is log canonical, 
using the ACC of log canonical threshold \cite{HMX}. 
It also follows from Lemma \ref{lem:idseq} applied at large enough $q$, we know 
the test configuration $\mathcal{Y}|_c$ is $\mathbb{Q}$-Gorenstein family with 
at worst semi-log-canonical central fiber. 
The remained step of the proof of Proposition \ref{corA} is to show that 
the fiber over the closed point $p=(\mathfrak{m},t)$ is a K-semistable 
klt cone as a special degeneration in the sense of \cite{LX14, LWX}. 
Further, note that 
the Donaldson-Futaki invariant 
(Definition \ref{def:cs} \eqref{def:df}) 
of the affine test configuration 
along $c\in S=\Spec(R)$ is $0$ 
as it is the same as that along the generic point 
$\eta\in S=\Spec(R)$. 
Thus it follows from 
\cite[\S 4, Proposition 4.3, Corollary 4.4]{LWX} 
that the fiber over the closed point $p=(\mathfrak{m},t)$ in $\mathcal{Y}$ 
is log terminal. 
The only remained part is to show the 
K-semistability of 
the fiber over the closed point $p=(\mathfrak{m},t)$. 
This follows the same arguments as 
\cite[Lemma 3.1]{LWX} (a special case of the natural cone version of 
the CM minimization conjecture), 
once we modify the proof therein verbatim by 
replacing the Hilbert scheme by 
the multi-Hilbert scheme and the Futaki invariant 
for the Fano varieties (corresponding to regular case)
by the Futaki invariant for the fixed positive 
vector field $\xi$. This completes the proof of 
Proposition \ref{corA}. 

\end{proof}


\subsection{S-completeness and its consequences}

In this subsection, 
we confirm another ingredient for the 
properties of the moduli stack 
$\mathcal{M}_{V}^{\rm ss}(\xi)$, after \cite{LWX}. 

First we review the following algebraic stack, 
which is convenient for the framework of \cite{HL, AHLH}. 

\begin{defn}[{\cite[\S 2B]{HL}}]
For a DVR $R$ with its uniformizer $\pi$, 
${\overline{\rm ST}}_R:=[\Spec (R[x,y]/(xy-\pi))/\G_m]$. 
Here, $\G_m$ acts on $x, y$ with weights $1,-1$ respectively. 
The closed point as the image of $(x,y)$ is denoted as 
$0$. 
\end{defn}

Then, we rephrase a theorem of \cite{LWX} as follows. 

\begin{Thm}[cf., \cite{LWX}]\label{thm:Sc}
For any fixed $n, V$ and $\xi$, 
$\mathcal{M}_{V}^{\rm ss}(\xi)$ is 
S-complete (over $\k$) in the sense of \cite[\S 3.5, 3.38]{AHLH}. 
That is for any morphism 
$\varphi^o\colon ({\overline{\rm ST}}_R\setminus 0)\to \mathcal{M}_{V}^{\rm ss}
(\xi)$ 
with essentially finite type DVR $R$ over $\k$, 
it extends to 
$\varphi\colon {\overline{\rm ST}}_R \to \mathcal{M}_{V}^{\rm ss}(\xi)$. 
\end{Thm}

\begin{proof}
This is essentially proved in \cite{LWX} 
(without the name of S-completeness in \cite{AHLH}). 
Take a $\G_m$-equivariant locally stable 
family of K-semistable Fano cone 
$\mathcal{X}^o\to \Spec (R[x,y]/(xy-\pi))$ 
which correpsponds to $\varphi^o$. This is equivalent to 
consider two special test configurations of its general fiber $X$. 
Then, 
\cite[proof of 4.1]{LWX} shows that it extends to 
(automatically $\G_m$-equivariat) faithfully flat affine 
$\Q$-Gorenstein 
family $\X\to \Spec (R[x,y]/(xy-\pi))$. 
Further, 
\cite[4.3, 4.4]{LWX} (after its 4.1) shows that the 
fiber over the closed point $(x,y)$ is K-semistable 
Fano cone, which gives rise to the desired extended morphism 
$\varphi$.     
\end{proof}
See also \cite{BX, ABHLX, LX, XZ2} 
for related work. 
\begin{cor}
For a K-polystable Fano cone 
$T\curvearrowright X, \xi$, 
its automorphism group 
${\rm Aut}(T\curvearrowright X, \xi)(\k)
:=\{T\text{-equivariant automorphism }X\to X\}$ 
forms a reductive algebraic $\k$-group 
${\rm Aut}(T\curvearrowright X, \xi)$. 
\end{cor}
\begin{proof}
This follows from the above S-completeness theorem  
\ref{thm:Sc} 
by 
\cite[3.47]{AHLH}. 
\end{proof}

\begin{Rem}
The reductivity also follows from more differential 
geometric arguments, simply 
combining \cite[Appendix]{DSII} and 
\cite[Theorem 2.9]{Li21}. 
\end{Rem}

In \S \ref{sec:existence}, 
we also use the above S-completeness to prove 
the separatedness of the moduli, following \cite[1.1]{AHLH}. 

\subsection{Properness}

We are now ready to prove the following theorem, 
which partially generalize the results of \cite{BHLLX, LXZ}. 

\begin{Thm}\label{Kmod.proper}
For any fixed $n, V$ and $\xi$, 
$\mathcal{M}_{V}^{\rm ss}(\xi)$ is universally closed 
(i.e., satisfies the existence part of the valuative criterion). 
\end{Thm}

\begin{proof}
As a direct application of Theorem \ref{A1}, 
we confirm that 
\begin{Claim}\label{higherTheta}
the filtration 
$\{\M_V(i)\}_i$ of \S \ref{sec:loccl} 
naturally holds the structure of the higher $\Theta$-stratification 
(in the sense of \cite[\S 3.3]{Od24b}, after \cite{HL}) for 
some rational polyhedral cone $\tau\subset N\otimes \R$ which contains  
$\xi$. 
\end{Claim}
\begin{proof}[proof of Claim \ref{higherTheta}]
We proceed to the proof of Claim \ref{higherTheta}. 
Take a point $b\in B_V(i)\setminus B_V(i-1)$ for $1\le i\le m'$ 
and consider the corresponding Fano cone $X_b$ and its 
Reeb vector field $\xi_b\in N_b\otimes \R$ where 
$X_b$ is acted by the a priori different  
algebraic subtorus $N_b\otimes \G_m=T_b$ of $G$, 
which commutes with $T$. 
Corresponding to the generalization of the extended Rees construction by 
Teissier \cite[\S 2.1, Proposition 2.3]{Teissier0}, 
we consider the 
$T_b$-equivariant morphism $\psi_b \colon 
\Delta_\xi^{\rm gl} \to \overline{T_b \cdot b}$ 
and induced 
$\overline{\psi_b} \colon 
[\Delta_{\xi_b}^{\rm gl}/T_b] \to 
[\overline{T_b\cdot b}/T_b]$ 
(see also \cite[\S 2.1]{LX}, 
\cite[Example 2.10, Theorem 2.12]{Od24b}). 
Now, take a connected component $B$ of 
$(B_V(i)\setminus B_V(i-1))_{\rm red}$ which means 
the closed subset 
$(B_V(i)\setminus B_V(i-1))$ 
with the reduced scheme structure, 
regarded as a closed subscheme. Then, $N_b$, $T_b$ and $\xi_b$ only depends on 
$B$ by \cite{Chen}, so that we denote them by 
$N_B, T_B, \xi_B$ respectively. 
Recall \cite[Lemma 3.14]{Od24b} which in particular says that any element of 
\begin{align}
&{\rm Map}([\Delta_{\xi_B}^{\rm gl}/T_B], \mathcal{M}_V(i))(B)\label{lhs} \\ 
=&\varinjlim_{N_B\otimes \R\supset \tau_B\ni \xi} {\rm Map}
([\Theta_{\tau_B}/T_B],\mathcal{M}_V(i))(B)\label{rhs}
\end{align}
comes from 
${\rm Map}
([\Theta_{\tau_B}/T],\mathcal{M}_V(i))(B) \text{ for some }\tau_B$. 
By Claim \ref{Claim:uniform}, 
there is an element of the left hand side \eqref{lhs} which induces 
degenerations to K-semistable Fano cones for any $b\in B$ simultaneously. 
Then, we use the above equality and represent by 
${\rm Map}
([\Theta_{\tau_B}/T],\mathcal{M}_V(i))(B)$ 
of the right hand side \eqref{rhs} for some $\tau_B$. We fix such $\tau_B$. 
Therefore, we obtain a $T_B$-equivariant morphism 
$\psi_B\colon U_{\tau_B} \times B \to B$ which induces $\overline{\psi}\colon \Theta_{\tau_B}\times B\to 
\M_V(i)$, 
where $U_{\tau_B}$ means the affine toric variety 
for $\tau_B$. 
$\psi$ localises to 
$(\Delta_\xi^{\rm gl}\times B) \to B_i$, 
which we still write as $\psi$. 
Take the connected component of 
the (finite type) algebraic stack 
${\rm Map}
([\Theta_{\tau_B}/T],\mathcal{M}_V(i))$ (cf., \cite{HL}, 
\cite[5.10, 5.11]{AHR}, \cite[6.23]{AHR2})
which contains the image of $B$ and denote as $\mathcal{B}$. 
Take any geometric $\k$-point $b'\in \mathcal{B}$ and consider 
the corresponding image of the vertex $p_{\tau_B}$ of $U_{\tau_B}$ as $[Y\subset \A^N]$. 
Then from the construction, its multi-Hilbert function remains the same as those parametrized by 
$B$. So, combined with the uniqueness of the 
$\vvol$-minimizer \cite{XZ2, BLnew}, 
it follows that 
the images of $p_{\tau_B}\times B$ are again inside $B_i\setminus B_{i-1}$ and 
the image of $\Theta_\tau\times \B\to \M_V$, 
which extends $\overline{\psi}$, 
still lies inside $\mathcal{M}_V(i)$. 
Therefore, 
we obtain the higher $\Theta_{\tau_m}$-stratification structure on each strata 
on $\mathcal{M}_V$. 
\end{proof}

Now, we proceed to the proof of Theorem \ref{Kmod.proper} by the valuative criterion. 
Take any DVR $R$ of essentially finite type over $\k$ (cf., \cite[Appendix A3]{AHLH}), 
and denote its quotient field (resp., residue field) as $K$ (resp., $\kappa$) and set 
$S:={\rm Spec}(R)$, its generic point $\eta$ (resp., closed point $c$). 
Consider an arbitrary morphism 
$\varphi\colon S\to B_V$ which maps $\eta$ into $B_V^{\rm ss}(\xi)$. 
We denote the corresponding family of Fano cones as 
$Y_K\to \eta$. 
Now we want to show that the induced morphism 
$\eta\to \M^{\rm ss}_V(\xi)$ can be extended to 
a morphism $S\to  \M^{\rm ss}_V(\xi)$. 
Take any common rational positive vector field $\xi'$ and corresponding 
algebraic subtorus $T'\subset T$ of rank $1$. Now, we apply \cite[Theorem 1]{LX14} $T$-equivariantly to 
a $\Q$-Gorenstein family of log Fano pairs 
$Y_K/T'\twoheadrightarrow \eta$ (cf., which underlies the associated Seifert $\G_m$-bundle's base cf., 
\cite{Kol13}) and take its cone: 
possibly after replacement of $R$ by its finite extension, 
this gives an extension $T\curvearrowright Y\twoheadrightarrow S$ of 
$(T/T')\curvearrowright Y_K/T'\twoheadrightarrow \eta$ 
as a $T$-equivariant family of 
Fano cones, corresponding to some morphism $S\to B_V(i)$ (after temporarily 
enlarging $B_V$ if necessary). 
Now, by Claim \ref{higherTheta}, 
we can use the higher $\Theta$-stable reduction theorem in 
\cite[\S 3.2, \S 3.3]{Od24b}, in the place of \cite{AHLH} for the original proof of 
the Fano varieties case \cite{BHLLX, LXZ} to decrease $i$, and hence finally reach $i=0$ by 
its repetition thanks to the ACC of normalized volumes \cite[1.2]{XZ24}. 
This completes the proof 
of Theorem \ref{Kmod.proper}. 
\end{proof}

\subsection{Existence of good moduli space}
\label{sec:existence}

\subsubsection{General statements and proof}

Now we are ready to construct the moduli spaces. 

\begin{Thm}[Proper K-moduli of Fano cone]\label{Kmod.Fcone}
For fixed $n, V$, there is a moduli Artin stack of $n$-dimensional 
K-semistable Fano cones of normalized volume $V$ and 
$\mathcal{M}_V^{\rm ss}(\xi)$, which admit a 
proper good moduli space $M_V^{\rm ss}(\xi)$ 
whose closed points parametrize K-polystable 
Fano cones among them. 
\end{Thm}

\begin{proof}
We constructed the 
moduli stack of 
K-semistable Fano cones of normalized volume $V$ and 
$\mathcal{M}_V^{\rm ss}(\xi)$ as \eqref{Mss}. 
Then by the $\Theta$-reductivity (Proposition \ref{corA} of 
Theorem \ref{A1})  
and the S-completeness theorem \ref{thm:Sc} 
(cf., \cite[\S 4]{LWX}) 
of $\mathcal{M}_V^{\rm ss}(\xi)$, 
\cite[Theorem A]{AHLH} can be applied to 
prove the existence of 
separated good moduli space $M_V^{\rm ss}(\xi)$. 

Furthermore, recall that the normalized volumes of 
$n$-dimensional log terminal singularities 
satisfy ACC \cite[1.2]{XZ24}. Thus, Theorem \ref{Kmod.proper} 
as an application of the 
higher $\Theta$-stable reduction theorem 
\cite[\S 3]{Od24b} can be applied to show that 
$M_V^{\rm ss}(\xi)$ is universally closed and hence 
proper. We complete the proof. 
\end{proof}

\begin{Rem}[Moduli ($2$-)functor]\label{family.subtle}
The set-theoretic meaning of the moduli 
$\mathcal{M}_V^{\rm ss}(\xi)(\k)$ and 
$M_V^{\rm ss}(\xi)(\k)$ 
are clear i.e., $n$-dimensional 
K-semistable (resp., K-polystable) 
Fano cones of volume density $V$ and the Reeb vector 
field $\xi$ for $T$, 
as we defined (in the previous section), 
but more generally  
$S$-valued points for more general $\k$-schemes $S$ i.e., 
$\mathcal{M}_V^{\rm ss}(\xi)(S)$ 
can be redefined more intrinsically as their 
``locally stable" families 
$T\curvearrowright \mathcal{X}\twoheadrightarrow S$ 
in the sense of Koll\'ar \cite[Definition 3.40]{Kol22}. 
The arguments are essentially the same as 
\cite{hull, Kol22, Xu}, so we omit the detail. 
\end{Rem}

\begin{Rem}[Generalization to logarithmic setup]\label{family.subtle.log}
It is also straightforward to generalize the above 
theorems and their proofs to kawamata-log-terminal 
log Fano cones $(X,\Delta)$, their families and 
moduli, at least when $\Delta$ are $\Q$-divisors and 
marked as 
$\Delta=\frac{1}{N}D$ with ($\Z$-)Weil divisors, 
over reduced\footnote{to avoid complication, 
which is treated by the notion of ``K-flatness" in 
\cite{Kol22}} base. 
More precisely, 
we fix an algebraic $\k$-torus $T=N\otimes \G_m$ and 
consider the groupoid of 
$T$-faithfully flat families $\mathcal{X}\xrightarrow{\pi} S$ of 
normal affine cones 
(in the sense of 
Definition \ref{def:cone} (ii)) 
together with relative Mumford 
$(\Q)$-divisors which are marked as 
$\frac{1}{l}\mathcal{D}$ 
with $l\in \Z_{>0}$ and relative Mumford
\footnote{which simply means being 
relative Cartier divisor 
inside the open subset of $\X^{\rm sm}\subset \X$, 
the relatively smooth locus, in our fiberwise normal  
case}
$\Z$-
divisors $\mathcal{D}$, which satisfy the same 
``$\Q$-Gorenstein-ness of deformation" type condition 
of \cite[8.13 (except for 8.13.5)]{Kol22}. 
We further restrict our attention to 
the cases with reduced $S$ and 
when $\pi$-fibers $(\X_s,\frac{1}{l}\mathcal{D}_s)$ 
for $s\in S$ are all log K-semistable with respect to 
$\xi$, and denote such groupoid as 
$\mathcal{M}_V^{\rm lss,red}(\xi, l)(S)$. 

Then, by the 
log boundedness of 
$\mathcal{M}_V^{\rm lss,red}(\xi, l)(\k)$ (\cite{XZ24}), 
we use the multi-Hilbert scheme \cite{HS} again of 
$T\curvearrowright X$ together with 
\cite[7.3]{Kol22} (applied to projective compactifications 
of $D$s with respect to a positive vector field $\xi'\in N$) 
to obtain a larger moduli stack of log affine cones, 
in the sense of Definition \ref{def:cone}. 
Note that the divisors $D$ are also assumed to be 
$T$-invariant which, 
again by the theory of \cite{HS}, 
provides the reason of locally finite typeness of the obtained stack, and its finite typeness is ensured by 
the log boundedness \cite{XZ24}. 
Then, by \cite[3.22]{Kol22} (also cf., 
\cite{hull}), we can take its locally closed substack 
generalize our previous construction of 
$\mathcal{M}_V^{\rm ss}(\xi)$. Then, the 
remained confirmation of its properties work 
verbatim after our discussions above in this \S 
\ref{sec:moduli} (and the materials of \S \ref{sec:prep}  
and the references 
used in the proofs), 
using the logarithmic framework of 
modern birational geometry (cf., \cite[\S 3]{KM}). 
Summarizing up, our 
arguments in this paper gives the following, 
similarly to Theorem \ref{Kmod.Fcone}: 
\begin{Thm}[K-moduli of log Calabi-Yau cones]\label{Kmod.Fcone2}

We fix $n,l \in \Z$, and $V\in \R_{>0}$. 
Then, 
by the logarithmic version of the 
boundedness with fixed $n,N,V,l$. 
(\cite{Jiang, XZ24}, cf., also Corollaries \ref{bdcor1}, \ref{bdcor2}), 
there are only finitely many choices of 
$N, T=N\otimes G_m, 
\xi \in N\otimes \R$ 
of $n$-dimensional 
log K-semistable log Fano cones 
$(T\curvearrowright X\ni x,\Delta=\frac{1}{l}D)$ 
where $(X,\Delta)$ are kawamata-log-terminal and 
are of the normalized volume $V$. 

Further fixing $N, T=N\otimes G_m, 
\xi \in N\otimes \R$ (just for convenience of description), 
they form a moduli Artin stack 
$\mathcal{M}_V^{\rm lss,red}(\xi, l)$ 
at the reduced Artin stack level, which admits a reduced 
proper good moduli space $M_V^{\rm lss,red}(\xi, l)$ 
whose closed points parametrize 
$n$-dimensional 
log K-polystable log 
Fano cones $(X,\frac{1}{l}D)$ 
among them. 
\end{Thm}
In particular, 
if $\k=\C$, 
$M_V^{\rm lss,red}(\xi, l)(\C)$ 
can be again regarded as 
the moduli of Sasaki-Einstein manifolds with 
certain singularities including ``conical" type 
which we do not give intrinsic formulation in this paper. 
\end{Rem}

\subsubsection{Reconstructing Fano varieties' K-moduli}
For the case when $\xi\in N$ 
and provides regular positive (Reeb) vector field, 
the above theorem \ref{Kmod.Fcone} reproves the following. 

\begin{cor}[K-moduli of $\Q$-Fano varieties]\label{Kmod.F3}
For fixed $n$ and fixed $V$, there is a moduli stack 
$\mathcal{M}$ of $n$-dimensional 
K-semistable $\Q$-Fano varieties $X$ of 
the anticanonical volume $V=(-K_X)^n$, and it further 
admits a 
proper good moduli space $M_V^{\rm ss}(\xi)$ 
whose closed points parametrize K-polystable 
$\Q$-Fano varieties. 
\end{cor}

We make further discussions on the possibility of 
yet other (general) apporoaches to algebraic 
construction of 
K-moduli of Fano varieties, in the next section 
\S \ref{sec:discussions}.

\begin{prop}
We fix $n, V, T, \xi$ below. 
The (universal) 
CM $\R$-line bundle $\lambda_{\rm CM}(T\curvearrowright \X)$ 
(Definition \ref{CMline})
on the moduli stack $\mathcal{M}_V^{\rm ss}(\xi)$ 
of $n$-dimensional Fano cones of the volume density $V$ 
for the Reeb vector field $\xi$ descends to a $\R$-line bundle on the 
good moduli space $M_V^{\rm ss}(\xi)$ which we denotes as 
$L_{\rm CM}(n,V,\xi)$ for simplicity. 
\end{prop}

\begin{proof}
The proof is essentially the same arguments as 
\cite[\S 6.2, after (K-moduli) Conjecture 6.2]{OSS}. 
Indeed, since the morphism 
$\mathcal{M}_V^{\rm ss}(\xi)\to M_V^{\rm ss}(\xi)$ 
is \'etale locally a GIT quotient as in the discussion of \cite{OSS}, 
the descendability is equivalent to the vanishing of the Donaldson-Futaki invariants for product test configurations for every K-semistable 
Fano cones, which follows from the definition of 
K-semistability (\cite{CS}, our Definition \ref{def:cs}). 
\end{proof}

\vspace{3mm}

\subsubsection{Projectivity and CM $\R$-line bundle}
Naturally, as conjectured in polarized setup in \cite{FS90}, \cite{Od10}, 
\cite[6.2]{OSS} and algebraic parts solved by \cite{CP, XZ} for Fano 
varieties case, we conjecture: 

\begin{conj}\label{projconj}
For $\k=\C$, $c_1(M_V^{\rm ss}(\xi)(\C),L_{\rm CM}(n,V,\xi))$ is represented by a 
Weil-Petersson type K\"ahler current on 
$M_V^{\rm ss}(\xi)^{\rm an}=M_V^{\rm ss}(\xi)(\C)$. 
For general $\k$, 
$L_{\rm CM}(n,V,\xi)$ is an ample $\R$-line bundle 
so that $M_V^{\rm ss}(\xi)$ is projective.     
\end{conj}

We omit the precise form of the 
natural log extension of the CM line bundle and 
the corresponding positivity conjecture, 
which the readers can guess from the 
polarized varieties case (cf., e.g., \cite[\S 6]{OSS}, 
\cite[\S 3]{CP}). 

\vspace{3mm}

\subsection{Examples}\label{sec:ex}

\subsubsection{Quasi-regular case}

In the quasi-regular case, as it is naturally expectable  
after \cite{MSY1, MSY2, CS, CS2}, 
we can reduce the study to the K-moduli of 
(proper) log Fano varieties. This generalizes 
Corollary \ref{Kmod.F3}, which corresponds to the regular 
Reeb vector field case. 

\begin{prop}
The moduli stack $\mathcal{M}_V^{\rm ss}(\xi)$ 
(resp., $M_V^{\rm ss}(\xi)$) 
of $n$-dimensional Fano cones $(T=\G_m)\curvearrowright X$ of volume density $V$ 
for quasi-regular Reeb vector field $\xi\in N_{\Q}$ is 
isomorphic to the K-moduli stack (resp., 
K-moduli space) of certain log Fano varieties 
which appears as quotient of $X$. 
\end{prop}

\begin{proof}
Take any $\G_m$-equivariantly faithfully flat 
family of affine cones $\G_m \curvearrowright\X=\cup_s X_s \subset 
\A^N\otimes S\to S\ni s$ 
with the weight vector $\xi=(m_1,\cdots,m_N)$, 
where $m_i$ are coprime positive integers. 
By diagonalization, 
any $\G_m$-equivariantly faithfully flat 
family of quasi-regular affine cones 
can be localized to the above type family. 
As noted by \cite{CS, CS2} (cf., also \cite{LL}), 
the K-(semi/poly)stability of $X_s$ is equivalent to that of $(X_s\setminus x)/\G_m$ 
with $\sum_D \frac{m_D-1}{m_D}D$ 
where $D$ runs over prime divisors inside 
$(z_i=0)\subset (X_s\setminus x)/\G_m$ and 
$m_D$ is the ramification degree 
of $[(X_s\setminus x)/\G_m]\to (X_s\setminus x)/\G_m$ 
i.e., $m_D={\rm gcd}(m_1,\cdots,m_{i-1},m_{i+1},\cdots,m_N)$.     

If we consider the K-moduli stack $\mathcal{M}$ 
of log Fano varieties 
$((X_s\setminus x)/\G_m, 
\sum_D \frac{m_D-1}{m_D}D)$, the above arguments give us a 
morphism 
$\varphi\colon \mathcal{M}_V^{\rm ss}(\xi)\to \mathcal{M}$. 
The inverse of this morphism also exists 
by taking the relative cones. 
Hence, these moduli stacks are isomorphic. 
Further, because of the fact that the K-polystability 
of $X$ and $\varphi(X)$ are equivalent, 
or otherwise from the existence of proper good moduli spaces for both, 
we conclude that $\varphi$ 
also induces an isomorphism at the 
good moduli space level. 
\end{proof}

Therefore, we can reduce the following to the 
log K-stability and log K-moduli 
of (usual) log Fano varieties. 

\begin{Prob}
    Explicitly describe the structure of K-moduli 
    (Theorem \ref{Kmod.Fcone})
    of 
quasi-regular Fano cones (Sasaki-Einstein manifolds with 
singularities) 
in the case of \cite{BGN, BGKollar, Kol05, LST}. 
\end{Prob}

\subsubsection{Irregular case}
The examples of 
compact {\it irregular} Sasaki-Einstein 
manifolds 
found in the initial stage 
seem to be \cite{G04a}, \cite{G04b}, 
\cite{MS}, \cite{FOW} among others 
and they are mostly toric 
Fano cones, so that they do not have 
positive dimensional moduli spaces. 

Recently, 
S\"u\ss \cite{Suss} considered an example of 
$3$-dimensional affine $T$-varieties 
of complexity one (cf., \cite{IltenSuss}) 
which admit a positive dimensional moduli space. 
The example in {\it loc.cit} 
is constructed as follows. 

By the correspondence of 
normal affine $T$-varieties 
with poloyhedral divisors \cite{AH06}, 
if we consider $Y=\mathbb{P}^1$ and 
a polyhedral divisor (\cite{AH06}) 
i.e., $\mathcal{D}=\sum_{y\in \P^1(\k)} \mathcal{D}_y$ 
where $\mathcal{D}_y$ are rational polyhedra of 
the same tail cone $\sigma$ which satisfy 
``proper"ness condition of \cite[\S 1]{AH06}. 

In \cite[\S 6]{Suss}, 
disjoint $y_1,\cdots,y_k \in \P^1(\k)\setminus \{0,\infty\}$ are fixed and we set 
$\sigma$, 
$\mathcal{D}_y$ as: 
\begin{itemize}
\item $\sigma:=\R_{\ge 0}(-1,1)+\R_{\ge 0}(15k-4,8)$, 
\item $\mathcal{D}_0:=(\frac{2}{5},\frac{1}{5})+\sigma$, 
\item $\mathcal{D}_{\infty}:=(\frac{-2}{3},\frac{1}{3})+\sigma$, 
\item $\mathcal{D}_{y_i} (i=1,\cdots, k):=
([0,1]\times \{0\})+\sigma$, 
\item $\mathcal{D}_{y}:=\sigma$ otherwise i.e., 
when $y\neq 0,\infty, y_i (i=1,\cdots,k)$. 
\end{itemize}
In \cite[\S 6]{Suss}, 
the corresponding $T$-variety $X_k$ is proven to be 
K-polystable Fano cone for any positive integer $k$. 
Due to the automorphism group $\G_m(\k)=\k^*$ of 
$(\P^1,(0)+(\infty))$, we have 
the corresponding moduli 
\begin{align}\label{Suss..moduli}
M_k(\k)=(\P^1(\k)\setminus \{0,\infty\})^{k}/(\G_m(\k)=\k^*)\simeq \k^{k-1}.
\end{align}
To explicitly understanding our compactification 
(Theorem \ref{Kmod.Fcone}) in this case, 
now we would like to consider the case when 
$y_i$s can collide and can be even 
$0$ or $\infty$. When the supports of some $y$ collide, 
we take the Minkowski sum of the original $\mathcal{D}_y$s.

\subsubsection*{Case 1}
Firstly, we consider the effect of collisions of 
$y_i$ i.e., $y_i=y_j$ outside $0,\infty$ 
for some $1\le i\neq j\le k$. 

By \cite{IltenSuss}, \cite[Definition 4.2, Proposition 
4.3]{Suss}, special test configuration of the obtained 
$T$-variety $X$ corresponds to an admissible 
point $y\in \P^1$ in the sense of {\it loc.cit}. 
Further, by \cite[Theorem 4.10]{Suss}, 
the K-polystability of $(X,\xi)$ remains 
equivalent as far as $\sigma_y$ in {\it loc.cit} \S 4 
and ${\mathbf{u}_y}$ does not change. 
Thus, in particular, $X$ is K-polystable 
if $y_i$s are all in $\P^1\setminus \{0,\infty\}$ 
even if they collides. 

\subsubsection*{Case 2}
Next we consider when some of 
$y_i$ collide with $0$. 
Suppose exactly $m (0\le m\le k)$ of $y_i$s 
attain $0$. Then, 
$\mathcal{D}_0$ is changed to 
$(\frac{2}{5},\frac{1}{5})+([0,m],0)\sigma$. 

Consider the case with $k=2$ and positive $m$. 
A lengthy hand calculation 
(after \cite[4.10, \S6]{Suss}), which we omit here, 
and confirmation by 
using the Sage package 
``TVars" created by Absar Gull, Leif Jacob, Leandro Meier and
Hendrik S\"u\ss, 
shows the Donaldson-Futaki invariant of the 
special test configuration for the admissible 
$y=0$ is negative. 
The author thanks the creators of the Sage package, 
especially Hendrik 
S\"u\ss for the help on this computation.  

\subsubsection*{Case 3}
Next we consider when some of 
$y_i$ collide with $\infty$. 
Then, 
if $k=2$ and $m'$ is positive, 
a similar calculation as above 
(after \cite[4.10, \S6]{Suss} and 
the same Sage program) 
shows the Donaldson-Futaki invariant of the 
special test configuration for the admissible 
$y=\infty$ is positive if $m'=1$ and negative 
if $m'=2$. 

The conclusion is that 
\begin{prop}\label{Suss.moduli}
In the generalized setup of \cite[\S 6]{Suss} 
(allowing $y_i$ to be $0,\infty$ and their 
collision), suppose $k=2$ and 
$y_i=0$ for $m$ $i$s and $y_i=\infty$ 
for $m'$ $i$s. Then, 
the obtained Fano cone $X=X_k$ is K-polystable 
if and only if $m\le 1$. 
\end{prop}

\begin{cor}\label{Suss.moduli2}
For $k=2$, the normalization of the 
K-moduli compactification $M_k=M_2$ 
(Theorem \ref{Kmod.Fcone}) of 
the above moduli \ref{Suss.moduli} (\eqref{Suss..moduli}) 
of the $3$-dimensional 
$T$-variety Fano cone of complexity $1$ 
is $\P^1=\P(1,2)$.     
\end{cor}

\begin{proof}[proof of Corollary \ref{Suss.moduli2}]
We consider $\frac{1}{y_i}=:x_i$ and 
take $[x_1+x_2:x_1 x_2]\in \P(1,2)=(\A^2_{x_1+x_2,x_1 x_2}\setminus 
(0,0))/\G_m$. 
The natural relative version of the construction of $T$-
variety in \cite{AH06} gives 
a family of $T$-variety $X$ over the above 
$\A^2_{x_1+x_2,x_1 x_2}\setminus (0,0)
=(\A^2_{x_1,x_2}\setminus (0,0))/S_2$. 
By combining with 
Proposition \ref{Suss.moduli}, 
we obtain the proof. 
\end{proof}

By similar computations and some experiments, 
we expect that for at least $k=3,4$ and 
possibly other small $k$s, 
that the normalization of the 
K-moduli compactification of 
$M_k$ is isomorphic to the log K-moduli of 
$$(\mathbb{P}^1,(1-c)[0]+c[y_1]+\cdots+c[y_k]+(1-3c)[\infty])$$ for distinct $y_i$s and $0<c\ll 1$ 
and their log K-polystable degenerations. 

\vspace{3mm}
It would be also interesting to work on these 
moduli from 
quiver-gauge theoretic side on the ADS-CFT 
correspondence. 

\section{Discussions}\label{sec:discussions}

\subsection{Via $\gamma$-invariants of Berman}

The reason of the name of 
``$\delta$"-invariant in \cite{FO} 
comes from the original work of Berman \cite{Berman}, 
where he introduces ``$\gamma$"-invariant 
$\gamma(X)\in \R_{> 0}$ for $\Q$-Fano varieties $X$ 
from his 
original probabilistic (or statistical mechanical) approach 
to the complex Monge-Amp\`ere measures of the 
K\"ahler-Einstein metrics. See \cite[\S 2.2]{FO} 
for the more algebraic explanation and comparision of 
the two invariants ($\gamma$ vs $\delta$). 
We also recall that 
\begin{Thm}[{\cite[Theorem 2.5]{FO}}]
$\delta(X)\ge \gamma(X)$ for 
any $\Q$-Fano variety $X$.   
\end{Thm}

See \cite{Berman} for the discussions on the 
analytic aspects of the possibility that the above 
can be equality in general. 
We point out below that the theory of 
$\gamma(X)$ a priori gives a direct 
algebraic construction of K-stable limit of 
Fano varieties without the iterative procedure 
\cite{AHLH, Od24b} nor some invariants a priori, 
as in the method via $\delta(X)$. 

\begin{prop}[K-stable Fano via log KSBA theory]
Consider any flat $\Q$-Gorenstein (locally stable \cite{Kol22}) 
family of 
$\Q$-Fano varieties $\X\xrightarrow{\pi} \Delta$, 
where $\Delta$ is a smooth $\k$-curve, 
which satisfies $\gamma(X_t=\pi^{-1}(t))>1$ 
(``uniformly Gibbs stable") 
for all $t\in \Delta$. Then, 
$\X^{(m)}:=\X
\underbrace{
\times_{\Delta}\X \times_{\Delta}\times 
\cdots \times
_{\Delta}}_{m\text{-times}}
\X$ is obtained as the relative log canonical model of 
$(\widetilde{\X}^{(m)},(1+\epsilon)(f^{(m)})^*\mathcal{D}_m/m)$ 
over $\Delta$, 
for any $m\gg 0$, 
where $f\colon \widetilde{\X}\to \X$ is any log 
resolution of $\X$, 
$\widetilde{\X}^{(m)}:=\widetilde{\X}
\underbrace{
\times_{\Delta}\widetilde{\X} \times_{\Delta}\times 
\cdots \times
_{\Delta}}_{m\text{-times}}
\widetilde{\X}$ which contracts as 
$\widetilde{\X}^{(m)}\xrightarrow{f^{(m)}}
\widetilde{\X}$ 
and $\mathcal{D}_m$ is the 
relative $m$-plurianticanonical divisor 
crucially constructed in \cite[\S 6]{Berman}. 
In particular, $\X$ can be recovered as its relative 
diagonal. 
\end{prop}

Recall that standard birational geometric arguments 
easily imply 
the above construction 
does not depend on 
the choice of $f$. 
Thus, the main point of the above observation is 
its construction of K-(semi)stable filling $\X_0$ 
does not depend on the choice of divisors etc, 
nor any kind of iterative procedures, 
since $\mathcal{D}_m$ 
is taken as a canonical object (for each $m$). 
Hence, this approach is close to that of the Koll\'ar-Shepherd-Barron-Alexeev for the (log) K-ample case (cf., \cite{Kol22}), 
which we now can interpret as a part of (log) K-stability 
theory (cf., the old survey \cite{Od10} and references 
therein).

\subsection{Via $\delta$-invariants}

Recall that 
the original proof in \cite{BHLLX, LXZ} 
crucially uses the following refinement of 
$\delta$-invariant: 

\begin{defn}[{\cite[(1.2)]{BHLLX}}]
For each log terminal $\mathbb{Q}$-Fano variety $X$, 
$$M^{\mu}(X):=\biggl(\delta(X),\quad \inf_{\X}\dfrac{\rm DF(\X,-K_{\X/\P^1})}{||\X||_{L^2}}\biggr),$$
as an element of $\R_{\ge 0}^2$ to which we put the lexicographic order. 
Here, 
$\delta(X)$ denotes the $\delta$-invariant of $(X,-K_X)$ (\cite{FO,BJ}), 
$\X$ runs over special test configurations of $X$, 
$||\X||_{L^2}$ is 
the $L^2$-norm of 
test configuration $(\X,-K_{\X/\P^1})$. 
\end{defn}

The crux of the original 
algebraic proof of this theorem is that 
we can modify the family $\mathcal{X}$ or its corresponding family 
$f\colon \Delta\to \mathcal{M}$ to improve $M^{\mu}(f(0))$ to make the first component 
$\delta(f(0))$ at least $1$. 

Comparison of the $\delta$-invariant (\cite{FO, BJ}) and the  
normalized volumes (\cite{Li, Li2, LL}) 
are rather nontrivial 
problems, 
although they are analogous (\cite{Liu}, 
\cite[Remark 4.5]{LLX}). 
For instance, the following is known. 

\begin{Thm}[cf., {\cite{Liu},\cite{BBJ},\cite{CRZ},\cite{BJ}}]
For any $\Q$-Fano variety $X$, we have 
\begin{align}
\delta(X,-K_X)^n (-K_X)^n\le \biggl(1+\frac{1}{n}\biggr)^n \vvol(x\in X).\label{strongdelta}
\end{align}
A weaker version is that 
\begin{align}
\min\{\delta(X,-K_X),1\}^n (-K_X)^n\le  \biggl(1+\frac{1}{n}\biggr)^n \vvol(x\in X).
\label{weakdelta}
\end{align}
\end{Thm}

\begin{proof}
The weaker version \eqref{weakdelta}
follows from the combination of the fact that 
$\min\{\delta(X), 1\}$ is the greatest Ricci lower bound 
(\cite[\S 7.3]{BBJ}, \cite[Appendix]{CRZ}) and the logarithmic 
analogue \cite[Proposition 4.6]{LL}
of \cite[Theorem 2]{Liu}, applied to 
$(X,(1-\min\{\delta(X,-K_X),1\})D)$. 
The stronger version \eqref{strongdelta}
also directly follows from \cite[Theorem D]{BJ} if we set $L=-K_X$ 
for $X$. 
\end{proof}

On the other hand, 
recently the theory of the $\delta$-invariants are also 
generalized to the (possibly irregular) 
Fano cone setups in a similar manner by \cite{Wu, Huang}. 
Note that \cite[7.0.1]{Huang} 
and the fact that minimized 
normalized volume can be irrational would 
imply that they are different in general. 
A natural question is  
\begin{Ques}
Can we use the above-mentioned  
generalizations of 
$\delta$-invariant (\cite{Wu, Huang}) 
to also give another 
properness proof of K-moduli of Fano cones, 
extending the approach in \cite{BHLLX, LXZ}? 
\end{Ques}

We plan to discuss versions of our results for 
K\"ahler-Ricci solitons in a different paper, 
which can be often considered as 
special cases of the Sasaki-Einstein manifolds 
(cf., \cite[Conjecture 1.2]{MN}). 


\subsubsection*{Acknowledgement}
The author thanks R.Berman, 
S.Sun, H.S\"u\ss,
and 
C.Spotti for the helpful conversations and 
inspirations related to this paper. 
He also thanks Z.Chen, M.Jonsson, M.Yamazaki for 
comments on the manuscript. 
During the research, the author was partially 
supported by Grant-in-Aid for Scientific Research (B) 23H01069, 21H00973, 
Grant-in-Aid for Scientific Research (A) 21H04429, 
20H00112, 
and Fund for the Promotion of Joint International Research (Fostering Joint International Research) 23KK0249. 
The last part of the work is done during the author's stay 
at Institut de math\'ematiques de Jussieu, Paris Rive 
Gauche in Paris. We appreciate 
all of them who helped our pleasant stay 
for the kind hospitality. 



\vspace{5mm} \footnotesize \noindent
Contact: {\tt yodaka@math.kyoto-u.ac.jp} \\
Department of Mathematics, Kyoto University, Kyoto 606-8285. JAPAN \\

\end{document}